\documentclass[11pt]{amsart}
\usepackage{amssymb,amsfonts}
\usepackage{euscript,mathrsfs}
\usepackage{xspace}
\usepackage{color}
\usepackage{xcolor}
\usepackage{cite}
\usepackage{cleveref}
\usepackage{graphicx}
\usepackage[utf8]{inputenc}
\usepackage{amsmath,amsthm}
\usepackage{url}
\usepackage{enumitem}
\usepackage[noend]{algpseudocode}
\usepackage{algorithm}
\usepackage{subcaption}
\graphicspath{{figures/}}
\theoremstyle{plain}
\newtheorem{theorem}{Theorem}[section]
\newtheorem{lemma}[theorem]{Lemma}
\newtheorem{proposition}[theorem]{Proposition}
\newtheorem{corollary}[theorem]{Corollary}
\newtheorem{problem}[theorem]{Problem}

\theoremstyle{definition}
\newtheorem{definition}[theorem]{Definition}

\theoremstyle{remark}
\newtheorem*{remark}{Remark}

\newcommand{\RR}{\mathbb{R}}

\newcommand{\Lag}{L}
\renewcommand{\SS}{\mathbb{S}}
\newcommand{\id}{I}
\newcommand{\ball}[2]{\mathbf{B}_{\mkern-1mu#2\mkern-1mu}(#1)}
\newcommand{\eqsdp}{$(\overline{\mathit{SDP}})$\xspace}
\newcommand{\eqls}{$(\overline{\mathit{SDP}}_{\!\mathit{ls}\!})$\xspace}

\DeclareMathOperator{\rank}{rank}
\DeclareMathOperator{\codim}{codim}
\DeclareMathOperator{\image}{Im}
\DeclareMathOperator{\dist}{dist}

\DeclareMathOperator{\tube}{tube}
\DeclareMathOperator{\poly}{poly}
\DeclareMathOperator{\mineig}{\min\mkern1mueig}

\begin{document}

\title[Polynomial time guarantees for the Burer-Monteiro method]
{Polynomial time guarantees for the Burer-Monteiro method}

\author{Diego Cifuentes}
\address{Massachusetts Institute of Technology \\ Cambridge, MA, USA}
\email{diegcif@mit.edu}

\author{Ankur Moitra}
\address{Massachusetts Institute of Technology \\ Cambridge, MA, USA}
\email{moitra@mit.edu}


\begin{abstract}
  The Burer-Monteiro method is one of the most widely used techniques for solving large-scale semidefinite programs (SDP).
  The basic idea is to solve a nonconvex program in $Y$,
  where $Y$ is an $n \!\times\! p$ matrix such that $X \!=\! Y Y^{T\!}$.
  In this paper, we show that this method can solve SDPs in polynomial time in a smoothed analysis setting.
  More precisely, we consider an SDP whose domain satisfies some compactness and smoothness assumptions,
  and slightly perturb the cost matrix and the constraints.
  We show that if $p \gtrsim \sqrt{2(1{+}\eta)m}$,
  where $m$ is the number of constraints and $\eta\!>\!0$ is any fixed constant,
  then the Burer-Monteiro method can solve SDPs to any desired accuracy in polynomial time, in the setting of smooth analysis.
  Our bound on~$p$ approaches the celebrated Barvinok-Pataki bound in the limit as $\eta$ goes to zero,
  beneath which it is known that the nonconvex program can be suboptimal.

  Previous analyses were unable to give polynomial time guarantees for the Burer-Monteiro method,
  since they either assumed that the criticality conditions are satisfied exactly,
  or ignored the nontrivial problem of computing an approximately feasible solution.
  We address the first problem through a novel connection with tubular neighborhoods of algebraic varieties.
  For the feasibility problem we consider a least squares formulation,
  and provide the first guarantees that do not rely on the restricted isometry property.
\end{abstract}

\maketitle

\section{Introduction} \label{sec:introduction}

Consider a \emph{semidefinite program} (SDP) in the space of $n{\times} n$ symmetric matrices~$\SS^n$, involving $m$ equality constraints:
\begin{equation} \label{eq:sdp}
  \tag{\textit{SDP}}
  \begin{gathered}
    \min_{X\in \SS^n}\; C\bullet X
    \quad\text{ such that }\quad
    \mathcal{A}(X) = b, \quad
    X \succeq 0,
  \end{gathered}
\end{equation}
where $C\!\in\! \SS^n$, $b\!\in\!\RR^m$ and $\mathcal{A}:\SS^n\!\to\!\RR^m$, $X \mapsto (A_1\!\bullet\! X,\dots, A_m\!\bullet\! X)$ is a linear map.
Though interior point methods can solve~\eqref{eq:sdp} in polynomial time,
they typically run into memory problems for large values of~$n$.
The Burer-Monteiro method~\cite{Burer2003,Burer2005} is one of the most widely used procedures for large scale problems.
Several papers have worked in understanding the practical success of this method, see e.g.,~\cite{Burer2005,Boumal2016,Boumal2018}.
Although several results have been shown,
they all fall short of showing that one can reach an approximately optimal solution of~\eqref{eq:sdp} in polynomial time.
In this paper we prove the first polynomial time guarantees for the Burer-Monteiro method,
under a compactness and smoothness assumption on the domain.

The Burer-Monteiro method consists in writing $X\!=\! Y Y^T$ for some $Y\in \RR^{n\times p}$,
and solving the following nonconvex optimization problem:
\begin{equation} \label{eq:bm}
  \tag{\textit{BM}}
  \begin{gathered}
    \min_{Y\in \mathfrak{M}}\;  C\bullet Y Y^T,
    \quad\text{ where }\quad
    \mathfrak{M} := \{ Y \in \RR^{n\times p} : \mathcal{A}(Y Y^T) = b \}.
  \end{gathered}
\end{equation}
Let $\tau(k) := \binom{k+1}{2}$ be the $k$-th triangular number.
It is known that problems~\eqref{eq:sdp} and~\eqref{eq:bm} have the same optimal value for any $p$ such that $\tau(p) \!\geq\! m$;
this is known as the \emph{Barvinok-Pataki bound}~\cite{Barvinok1995problems,Pataki1998}.
But due to nonconvexity, local optimization methods may converge to a critical point of~\eqref{eq:bm} which is not globally optimal~\cite{Waldspurger2018}.
In this paper we are mainly interested in 2nd-order critical points (abbreviated: 2-critical points).

Boumal et~al.~\cite{Boumal2016,Boumal2018} showed that~\eqref{eq:bm} has no spurious 2-critical points when $\tau(p) \!>\! m$,
assuming that the feasible set~$\mathfrak{M}$ is a smooth manifold and that the cost matrix~$C$ is \emph{generic}.
The result of Boumal et~al.\ gives a strong indication that,
for any $p$ above the Barvinok-Pataki bound,
a local optimization method for~\eqref{eq:bm} should lead to the global optimal of~\eqref{eq:sdp}.
However, there are two serious technical obstacles to derive polynomial time guarantees.

The first obstacle is that numerical algorithms must be terminated after finitely many iterations,
and hence the criticality conditions are not satisfied exactly.
This issue has been addressed by Pumir et~al.~\cite{Pumir2018},
though for values of~$p$ larger than the Barvinok-Pataki bound.
They introduce a smoothed analysis~\cite{Spielman2004} setting
in which the cost matrix is subject to a small random perturbation of magnitude~$\sigma$.
The purpose of this perturbation is to introduce genericity to the problem.
They then defined a notion of
\emph{exactly feasible approximately 2-critical} (EFAC) point,
and showed that any EFAC point of~\eqref{eq:bm} is also approximately optimal for~\eqref{eq:sdp} when
$\tau(p) \gtrsim \frac{9}{2}{m\log(\Omega(\!\sqrt{n}/\sigma))}$
with high probability.
Note that this bound gets worse when the perturbation magnitude~$\sigma$ decreases.

In this paper, we consider the same smoothed analysis setting (the cost $C$ is perturbed),
and we improve upon~\cite{Pumir2018} by matching the Barvinok-Pataki bound.
To do so, we provide a deterministic characterization of the spurious EFAC points of~\eqref{eq:bm}
in terms of tubular neighborhoods around algebraic varieties.
This deterministic characterization,
given in \Cref{thm:sufficient},
is of independent interest.
By using effective bounds for the volume of such tubular neighborhoods~\cite{Burgisser2008,Lotz2015,Gray2012,Basu2021}, we derive the following theorem.

\begin{theorem}[critical $\Rightarrow$ optimal $|$ Informal] \label{thm:informalQ1}
  Consider this setting:
  \begin{itemize}
    \item
      The rank~$p$ satisfies $\tau(p) > m$.
    \item
      Apply a perturbation of magnitude $\sigma$ to the cost matrix~$C$.
  \end{itemize}
  Then, with high probability, any EFAC point for~\eqref{eq:bm} with bounded norm is also approximately optimal for~\eqref{eq:sdp},
  provided that the criticality precision is sufficiently small.
  \emph{The precise statement appears in \Cref{thm:main}.}
\end{theorem}

The second obstacle we need to overcome
is the efficient computation of EFAC points.
Feasibility is the main impediment.
Indeed, the set $\mathfrak{M}$ is defined by quadratic equations, and solving quadratics is NP-hard.
Hence, finding EFAC points is computationally intractable.
To address this issue,
we relax the feasibility requirement and define a notion of
\emph{approximately feasible approximately 2-critical} (AFAC) point.
We show that \Cref{thm:informalQ1} remains valid for AFAC points.
Importantly, AFAC points can be computed in polynomial time (see \Cref{thm:complexity}),
provided that an approximately feasible solution is known.
This leads to the following theorem.

\begin{theorem}[Polytime optimality $|$ Informal] \label{thm:informalQ2a}
  Consider this setting:
  \begin{itemize}
    \item
      The rank~$p$ satisfies $\tau(p) > (1{+}\eta)m$,
      for a fixed constant~$\eta>0$.
    \item
      Apply a perturbation of magnitude $\sigma$ to the cost matrix~$C$.
    \item
      Assume that $\mathfrak{M}$ is compact and smooth (LICQ holds).
    \item
      Let $Y_0$ be an approximately feasible point,
      i.e., $Y_0$ is close to~$\mathfrak{M}$.
    \item
      Solve~\eqref{eq:bm} using a constrained optimization method with 2nd-order guarantees (e.g.,~\Cref{thm:complexity}) initialized at~$Y_0$.
  \end{itemize}
  Then, after $\poly(n,\sigma^{-1})$ iterations,
  the algorithm produces with high probability a point $Y$ such that
  $Y Y^T$ is approximately optimal for~\eqref{eq:sdp}.
  \emph{The precise statement appears in \Cref{thm:complete}.}
\end{theorem}

The above theorem applies to SDPs with compact and smooth domains
for which a feasible solution is known.
Observe that it requires a slightly larger bound for~$p$,
compared to \Cref{thm:informalQ1}.
This is needed to ensure that the criticality precision remains polynomially bounded.
Also note that the perturbation magnitude~$\sigma$ appears in the complexity of the algorithm (as is usual in smooth analysis),
but does not appear in the bound for~$p$.

In order to fully address the computation of AFAC points,
we need to find an approximately feasible solution of~\eqref{eq:sdp}.
To do so,
we consider the {least squares} problem
\begin{equation}\label{eq:ls}
  \tag{${\mathit{SDP}}_{\!\mathit{ls}\!}$}
  \min_{X \in \SS^n} \; \|\mathcal{A}(X)-b\|^2
  \quad\text{ such that }\quad
  X \succeq 0,
\end{equation}
and the corresponding Burer-Monteiro problem
\begin{equation}\label{eq:ls-bm}
  \tag{${\mathit{BM}}_{\!\mathit{ls}\!}$}
  \min_{Y \in \RR^{n\times p}} \; \|\mathcal{A}(Y Y^T)-b\|^2.
\end{equation}
We can analyze problem~\eqref{eq:ls-bm} using similar techniques as for~\eqref{eq:bm}.
We consider a smoothed analysis setting in which the constraint map~$\mathcal{A}$ is subject to a small perturbation.
We will show in \Cref{thm:ls-main} that, for $\tau(p) \!>\! m$, any approximately critical point of~\eqref{eq:ls-bm} is approximately optimal for~\eqref{eq:ls}.
This leads us to the following theorem.

\begin{theorem}[Polytime feasibility $|$ Informal] \label{thm:informalQ2b}
  Consider this setting:
  \begin{itemize}
    \item
      The rank~$p$ satisfies $\tau(p) > (1{+}\eta)m$,
      for a fixed constant~$\eta>0$.
    \item
      Apply a perturbation of magnitude $\sigma$ to the constraint map~$\mathcal{A}$.
    \item
      Assume that the perturbed set $\mathfrak{M}$ is nonempty and compact.
    \item
      Solve~\eqref{eq:ls-bm} using an unconstrained optimization method with 2nd-order guarantees (e.g.,~\cite{Cartis2011,Cartis2012}).
  \end{itemize}
  Then, after $\poly(n,\sigma^{-1})$ iterations,
  the algorithm produces with high probability a point $Y$ such that
  $Y Y^T$ is approximately optimal for~\eqref{eq:ls}.
  \emph{The precise statement appears in \Cref{thm:ls-complete}.}
\end{theorem}

\Cref{thm:informalQ2a,thm:informalQ2b} together provide polynomial time guarantees for the Burer-Monteiro method,
assuming that both $C,\mathcal{A}$ are slightly perturbed
and that the domain is compact and smooth.
Note that in certain applications a feasible point of the SDP is known,
in which case \Cref{thm:informalQ2a} alone is enough
(perturbing $\mathcal{A}$ is not needed).

We point out that problem \eqref{eq:ls} is a matrix sensing problem which is of interest in its own right, independent of its connection to SDP feasibility.
The guarantees from \Cref{thm:informalQ2b} apply even if the optimal value of \eqref{eq:ls} is strictly positive (i.e., there is no $X {\succeq} 0$ with $\mathcal{A}(X) {=} b$).
There are earlier works \cite{Bhojanapalli2016,Li2017,Ge2017} proving global optimality guarantees for the nonconvex problem \eqref{eq:ls-bm},
but they all rely on the restricted isometry property (RIP).
To the best of our knowledge, \Cref{thm:informalQ2b} provides the first global guarantees for \eqref{eq:ls-bm} that do not rely on~RIP.

The structure of this paper is as follows.
In \Cref{s:complexity} we introduce the notion of AFAC points in nonlinear programming,
and we discuss how to compute them in polynomial time.
In \Cref{s:optimality} we prove \Cref{thm:informalQ1}.
In \Cref{s:feasibility} we analyze the least squares problem~\eqref{eq:ls}.
In \Cref{s:complete} we put together the results from the paper,
and show \Cref{thm:informalQ2a,thm:informalQ2b}.
We conclude with some experimental results in \Cref{s:experiments}.

\subsection*{Related work}

The Burer-Monteiro method applied to problems~\eqref{eq:sdp} and~\eqref{eq:ls} has attracted much research in past years.
Several papers have tried to explain the practical success of the method,
as we elaborate now.

For problem~\eqref{eq:sdp}, early work by Burer and Monteiro~\cite{Burer2005} and by Journ\'ee et al.~\cite{Journee2010} gave strong indications that \eqref{eq:bm} has no spurious critical points above the Barvinok-Pataki bound.
A formal proof was given recently by Boumal et al.~\cite{Boumal2016,Bhojanapalli2018}.
Subsequent work by Bhojanapalli et al.~\cite{Bhojanapalli2018} and Pumir et al.~\cite{Pumir2018} investigated the case of approximately critical points in a smoothed analysis setting.
The work in~\cite{Bhojanapalli2018} did not focus on~\eqref{eq:sdp}, but in a penalized version.
The Barvinok-Pataki bound was recently shown to be optimal up to lower order terms for general SDPs~\cite{Waldspurger2018}.
However for structured families of SDPs a smaller rank might suffice~\cite{Bandeira2016,Mei2017}.
As discussed above, previous work has not yet shown polynomial time guarantees for problem~\eqref{eq:bm}.

Problem~\eqref{eq:ls} has been well studied in the matrix sensing community.
Bhojanapalli et al.~\cite{Bhojanapalli2016} showed that~\eqref{eq:ls-bm} has no spurious local minima under the RIP condition,
and they also provided polynomial time guarantees.
Similar results have been derived later, e.g.,~\cite{Li2017,Ge2017}.
Note that RIP is a very strong assumption about the condition number of the linear map~$\mathcal{A}$,
particularly since the RIP constant needs to be small~\cite{Zhang2018}.
In contrast, our result in \Cref{thm:informalQ2b} simply assumes a small perturbation around a worst-case instance~${\mathcal{A}}$.
This is a nondegeneracy condition, much weaker than RIP.

\section{Critical points in nonlinear programming}\label{s:complexity}

In this section we revisit the notion of critical points in unconstrained and constrained optimization.
We also discuss notions of approximately critical points,
and review methods with finite-time complexity guarantees.

\subsection{Unconstrained case}
Consider the optimization problem
\begin{equation} \label{eq:unconstrained}
  \tag{$P_{\rm un}$}
  \min_{y\in \RR^n}\; f(y),
\end{equation}
with $f:\RR^n \!\to\! \RR$ twice continuously differentiable.
A vector $y\in \RR^n$ is a \emph{2nd-order critical} point for~\eqref{eq:unconstrained},
abbreviated 2-critical,
if it satisfies:
\begin{gather*}
  \nabla f(y) = 0,
  \qquad\qquad
  \nabla^2 f(y) \succeq 0.
\end{gather*}
In the unconstrained case any local minimum of $f$ is also a critical point.

Practical optimization algorithms cannot obtain a solution satisfying the above equations exactly.
Hence, we consider a relaxation of these conditions.

\begin{definition}
  Given $\varepsilon_1,\varepsilon_2\in \RR_+$,
  we say that $y$ is $(\varepsilon_1,\varepsilon_2)$-\emph{approximately 2-critical} (AC) for \eqref{eq:unconstrained} if:
  \begin{gather} \label{eq:conditionsunconstrainedapprox}
    \|\nabla f(y)\| \leq \varepsilon_1,
    \qquad\qquad
    \nabla^2 f(y) \succeq -\varepsilon_2 \id_n.
  \end{gather}
\end{definition}

Several algorithms for unconstrained optimization with provable convergence guarantees are known.
Recent work has focused on deriving algorithms with finite-time guarantees.
In particular,
the trust region method computes an $(\varepsilon_1,\varepsilon_2)$-AC point in
$O(\max\{\varepsilon_1^{-2}\varepsilon_2^{-1},\varepsilon_2^{-3}\})$ iterations~\cite{Cartis2012},
and the adaptive regularization with cubics (ARC) method takes
$O(\max\{\varepsilon_1^{-2},\varepsilon_2^{-3}\})$ iterations~\cite{Cartis2011}.
We formally state the result for the ARC method.

\begin{theorem}[{\cite{Cartis2011}}] \label{thm:complexityunconstrained}
  Assume that there exists $\alpha>0$ such that:
  \begin{itemize}
    \item A point $y_0$ with $f(y_0)\leq \alpha$ is known.
    \item $f, \nabla f, \nabla^2 f$ are uniformly bounded and Lipschitz continuous on the lower set
      $ \{ y : f(y)\leq \alpha\} $.
  \end{itemize}
  The ARC method initialized at~$y_0$
  requires $O(\max\{\varepsilon_1^{-2},\varepsilon_2^{-3}\})$ iterations to produce
  an $(\varepsilon_1,\varepsilon_2)$-AC point $y$.
  Furthermore, each iteration requires $O(1)$ evaluations of $f$ and its derivatives.
\end{theorem}

\subsection{Constrained case}
Consider the nonlinear program
\begin{equation}\label{eq:nonlinear}
  \tag{$P_{\rm con}$}
  \min_{y\in \mathfrak{M}}\; f(y),
  \quad \text{ where }\quad
  \mathfrak{M} := \{ y \in \RR^m : h(y)=0 \},
\end{equation}
with $f:\RR^n \!\to\! \RR$, $h:\RR^n \!\to\! \RR^m$ twice continuously differentiable.
The Lagrangian function is $\Lag(y,\lambda) \!=\! f(y) \!+\! \lambda {\cdot} h(y)$.
A vector $y\!\in\!\RR^n$ is a \emph{2-critical point} if there are multipliers $\lambda\!\in\!\RR^m$ such that
\begin{gather*}
  h(y) = 0,
  \qquad\qquad
  \nabla_{\!y} \Lag(y,\lambda) = 0,
  \\
  u^T \nabla^2_{\!yy} \Lag(y,\lambda) u \geq 0,
  \quad\forall\, u \text{ such that }
  \nabla h(y) u = 0.
\end{gather*}
In the constrained case, for a local minima to be a critical point we need some regularity conditions.
One such condition is the \emph{linear independence constraint qualification} (LICQ),
that states that $\nabla h(y)$ is full rank.
This is equivalent to $\mathfrak{M}$ being smooth at~$y$
for the case of complete intersections (i.e., $\codim \mathfrak{M} = m$).

We now consider a relaxation of the criticality conditions.

\begin{definition}
  Given $\boldsymbol\varepsilon=(\varepsilon_0,\varepsilon_1,\varepsilon_2)\in \RR_+^3$, $\gamma\in\RR_+$,
  we say that $y$ is $(\boldsymbol\varepsilon,\gamma)$-\emph{approximately feasible approximately 2-critical} (AFAC) for \eqref{eq:nonlinear}
  if there exists $\lambda\in\RR^m$ such that:
  \begin{subequations}\label{eq:conditionsKKTapprox}
    \begin{gather}
      \label{eq:firstorderapprox}
      \|h(y)\| \leq \varepsilon_0,
      \qquad\qquad
      \|\nabla_{\!y} \Lag(y,\lambda)\| \leq \varepsilon_1,
      \\
      \label{eq:secondorderapprox}
      u^T \nabla^2_{\!yy} \Lag(y,\lambda) u \geq -\varepsilon_2,
      \quad\forall\, u \text{ of unit norm such that }
      \|\nabla h(y) u\| \leq \gamma.
    \end{gather}
  \end{subequations}
\end{definition}

Our subsequent analysis requires a bound on the Lagrange multipliers~$\lambda$.
We can obtain such a bound through a quantitative version of LICQ.

\begin{definition}\label{def:licqp}
  For $\varrho>0$, we say that \emph{$\varrho$-LICQ} holds at~$y$ if
  the smallest singular value of $\nabla h(y)$ is at least~$\varrho$.
\end{definition}

\begin{lemma} \label{thm:lambda}
  Let $(y,\lambda)$ be such that
  $\|\nabla_{\!y} \Lag(y,\lambda)\| \!\leq\! \varepsilon_1$.
  If $\varrho$-LICQ holds at~$y$,
  then $\|\lambda\| \leq \varrho^{-1}(\varepsilon_1 {+} \|\nabla f(y)\|)$.
\end{lemma}
\begin{proof}
  Let $J := \nabla h(y)$.
  Since
  $ \nabla L(y,\lambda) =  \nabla f(y) {+} J^T \lambda $,
  and $J$ is full rank,
  then
  $ \lambda = (J^\dagger)^T (\nabla L(y,\lambda) {-} \nabla f(y)) $,
  where $J^\dagger$ is the pseudo-inverse of $J$.
  Hence
  $
    \|\lambda\|
    \leq \varrho^{-1} (\varepsilon_1 {+} \|\nabla f(y)\|).
    \qedhere
  $
\end{proof}

Several local optimization methods for~\eqref{eq:nonlinear} with provable convergence guarantees are known.
In particular, augmented Lagrangians~\cite{Andreani2010} and trust-region methods~\cite[\S15.4]{Conn2000} converge to 2-critical points.
Recent work has focused on finding algorithms with finite-time guarantees.
The complexity of computing approximately 1-critical points was studied in, e.g.,~\cite{Birgin2016,Cartis2013,Cartis2014,Curtis2018}.

As for approximately 2-critical points, we are only aware of~\cite{Cartis2018,Sahin2019}.
But both papers use a different 2nd-order condition,
which is not easy to translate into our setting.
Nonetheless, in \Cref{thm:complexity} below
we show that AFAC points can be computed in polynomial time.
The proof of this theorem is in \Cref{s:outer},
and relies on a variant of the method from~\cite{Cartis2018}.
To the best of our knowledge,
this is the first polynomial time bound for computing 2-critical points that works with the standard notions of criticality.

\begin{theorem} \label{thm:complexity}
  Assume that there exist
  $\beta,\varrho \in \RR_+$
  such that:
  \begin{itemize}
    \item
      A point $y_0$ in the set
      $ \mathfrak{M}_\beta := \{ y : \|h(y)\| {\leq} \beta\}$
      is known.
    \item
      $f, \nabla f, \nabla^2 f$ and $\nabla h_i, \nabla^2 h_i$ for $i\in[m]$ are uniformly bounded and Lipschitz continuous on
      $ \mathfrak{M}_\beta$.
    \item
      $\varrho$-LICQ holds at all $y \in \mathfrak{M}_\beta$.
  \end{itemize}
  Let $\boldsymbol\varepsilon = (\varepsilon_0,\varepsilon_1,\varepsilon_2)\in \RR_+^3$,
  $\gamma \in \RR_+$
  be such that
  \begin{align}\label{eq:epsilonineqs}
    \varepsilon_0 \leq \beta,
    \quad
    \varepsilon_1 \leq 1,
    \quad
    \varepsilon_1^2 \leq \tfrac{1}{16} R_\lambda^{-1}\, \varepsilon_0\, \varepsilon_2,
    \quad
    \gamma^2 \leq \tfrac{1}{16} R_\lambda^{-3}\, \varepsilon_0\, \varepsilon_2,
  \end{align}
  where
  $R_\lambda \!:=\! 2 {+} 2\, \varrho^{-1} L_f$,
  $L_f \!:=\! \max_{y} \|\nabla f(y)\|$.
  There is an algorithm that,
  when initialized at~$y_0$,
  requires
  $O(\max\{\varepsilon_0^{-2}\, \varepsilon_1^{-2},\, \varepsilon_0^{-3}\, \varepsilon_2^{-3}\})$
  evaluations of $f, h$ and their derivatives to produce an
  $(\boldsymbol\varepsilon,\gamma)$-AFAC pair $(y,\lambda)$,
  with $\|\lambda\|\leq R_\lambda$.
\end{theorem}

\begin{remark}
  \Cref{thm:complexity} assumes that $\varrho$-LICQ holds everywhere in $\mathfrak{M}_\beta$.
  This is a strong assumption, that can only be guaranteed when $\beta$ is very small.
  Hence the initial point $y_0$ must be approximately feasible.
\end{remark}

\section{Optimality of critical points}\label{s:optimality}

In this section we consider a smoothed analysis setting in which the constraint variables $\mathcal{A},b$ are fixed,
and the cost matrix~$C$ is subject to a small random perturbation.
We will show that problem~\eqref{eq:bm} has no \emph{spurious} approximately critical (AFAC) points with high probability.
This means that any AFAC point of~\eqref{eq:bm} is approximately optimal for~\eqref{eq:sdp}.
We will restrict our attention to AFAC pairs~$(Y,\lambda)$ of bounded norm.
Hence, we assume that $\|Y\| \!\leq\! R_Y$ and $\|\lambda\| \!\leq\! R_\lambda$ for some fixed constants~$R_Y,R_\lambda$.

A crucial step toward our theorem is a geometric characterization of the spurious AFAC points in terms of tubes around algebraic varieties.
We then take advantage of known effective bounds for the volume of such tubes~\cite{Lotz2015,Basu2021}.

From now on, we use the Frobenius norm for all matrices.

\subsection{Spurious approximately critical points}

We proceed to describe the optimality conditions of problems~\eqref{eq:sdp} and~\eqref{eq:bm}.
Consider first the convex problem~\eqref{eq:sdp}.
Its conic dual is
\begin{align*}
  \max_{\lambda\in \RR^m}\; b^T \lambda
  \quad\text{ such that }\quad
  S(\lambda) \succeq 0,
\end{align*}
where
\begin{align*}
  S(\lambda) \;:=\; C - \mathcal{A}^*(\lambda) \;\in\; \SS^n\;\;
  \quad&\text{ is the \emph{slack} matrix,}
  \\
  \mathcal{A}^*:\RR^m \!\to\! \SS^n,\quad
  \lambda \!\mapsto\!\sum\nolimits_i \lambda_i A_i
  \quad&\text{ is the adjoint of }\mathcal{A}.
\end{align*}
The optimality conditions for the dual pair of SDPs are:
\begin{gather*}
  \mathcal{A}(X) = b,
  \quad
  X \succeq 0,
  \quad
  S(\lambda) \succeq 0,
  \quad
  S(\lambda)  X = 0.
\end{gather*}
The first three conditions correspond to primal and dual feasibility,
and the last one is complementary slackness.
We define an approximately optimal solution as a relaxation of the above conditions.

\begin{definition}
  Let $\boldsymbol\varepsilon\!=\!(\varepsilon_0,\varepsilon_1,\varepsilon_2)\!\in\! \RR_+^3$.
  We say that $X\!\in\!\SS^n$ is $\boldsymbol\varepsilon$-\emph{approximately optimal} for~\eqref{eq:sdp} if there is $\lambda\in\RR^m$ such that:
  \begin{gather} \label{eq:approxoptimal}
    \|\mathcal{A}(X) - b\| \leq \varepsilon_0,
    \quad
    \|S(\lambda)  X\| \leq \varepsilon_1,
    \quad
    X \succeq 0,
    \quad
    S(\lambda) \succeq -\varepsilon_2\, \id_n.
  \end{gather}
\end{definition}

It is known that an $\boldsymbol\varepsilon$-approximately optimal solution is at distance $O(\|\boldsymbol\varepsilon\|)$ from an optimal solution under nondegeneracy assumptions~\cite{Sturm2000}.
We can also give a simple bound on the optimality gap.

\begin{lemma} \label{thm:approxoptimal}
  If $(\bar X,\bar\lambda)$ is $\boldsymbol\varepsilon$-approximately optimal for~\eqref{eq:sdp} then
  \begin{align*}
    C\bullet \bar X \,\leq\, C\bullet X +\varepsilon_0 \|\bar\lambda\| + \varepsilon_1\sqrt{n} + \varepsilon_2 \|X\| \sqrt{n}
    \qquad\forall\text{ feasible } X.
  \end{align*}
\end{lemma}
\begin{proof}
  The lemma follows from the following equations:
  \begin{gather*}
    C\bullet X
    \,=\, \bar\lambda \cdot \mathcal{A}(X) + S(\bar\lambda)\bullet X
    \,\geq\, \bar\lambda \cdot b -(\varepsilon_2 \id_n)\bullet X
    \,=\, \bar\lambda \cdot b - \varepsilon_2 \|X\| \sqrt{n},
    \\
    \bar\lambda \cdot b
    \,\geq\, \bar\lambda \cdot \mathcal{A}(\bar X) - \|\bar\lambda\|\, \|b {-} \mathcal{A}(\bar X)\|
    \,\geq\, \bar\lambda \cdot \mathcal{A}(\bar X) - \varepsilon_0 \|\bar\lambda\|,
    \\
    \bar\lambda \cdot \mathcal{A}(\bar X)
    \,\geq\, \bar\lambda \cdot \mathcal{A}(\bar X) + S(\bar\lambda) \bullet X - \|S(\bar\lambda) X\|_*
    \,\geq\, C \bullet \bar X - \varepsilon_1 \sqrt{n}.
    \qedhere
  \end{gather*}
\end{proof}

We proceed to problem~\eqref{eq:bm}.
This is a special instance of the nonlinear program~\eqref{eq:nonlinear} with
$f(Y) = C\!\bullet\! Y Y^T$ and $h(Y) = \mathcal{A}(Y Y^T)\!-\!b$.
The Lagrangian function is
$\Lag(y,\lambda) = S(\lambda) \!\bullet\! Y Y^T \!+ b^T \lambda$.
The criticality conditions for~\eqref{eq:bm} are obtained by specializing~\eqref{eq:conditionsKKTapprox}.
\begin{definition}
Let $\boldsymbol\varepsilon \!=\! (\varepsilon_0,\varepsilon_1,\varepsilon_2) \!\in\! \RR_+^3$, $\gamma\!\in\!\RR_+$.
We say that $Y\!\in\!\RR^{n\times p}$ is $(\boldsymbol\varepsilon,\gamma)$-\emph{approximately feasible approximately 2-critical} (AFAC) for~\eqref{eq:bm} if there is $\lambda\!\in\!\RR^m$ such that:
\begin{subequations}\label{eq:approxcritical}
\begin{gather}
  \label{eq:conditionfirst}
  \|\mathcal{A}(Y  Y^T) - b\| \leq \varepsilon_0,
  \qquad
  \|S(\lambda)  Y \| \leq \varepsilon_1,
  \\
  \label{eq:conditionsecond}
  S(\lambda)\bullet U U^T \!\geq\! -\varepsilon_2,
  \quad\forall\, U\!\in\! \RR^{n\times p} \text{ with } \|U\|\!=\!1,\,
  \|\mathcal{A}(U  Y^T)\| \!\leq\! \gamma.
\end{gather}
\end{subequations}
\end{definition}

\begin{remark}
  By fixing $\varepsilon_0\!=\!0$ we get a notion of exactly feasible approximately 2-critical (EFAC) point,
  which was used in~\cite{Pumir2018,Bhojanapalli2018}.
\end{remark}

We are ready to formalize the concept of spurious critical points.

\begin{definition}
  Let $R_Y, R_\lambda \!\in\! \RR_+$ be fixed constants.
  For arbitrary $\boldsymbol\varepsilon \!=\! (\varepsilon_0,\varepsilon_1,\varepsilon_2) \!\in\! \RR_+^3$, $\gamma\!\in\!\RR_+$,
  we say that a pair $(Y,\lambda)$ is \emph{spurious} $(\boldsymbol\varepsilon,\gamma)$-AFAC if:
  \begin{itemize}
    \item
      $(Y, \lambda)$ is $(\boldsymbol\varepsilon,\gamma)$-AFAC for~\eqref{eq:bm}.
    \item
      $(Y Y^T\!, \lambda)$ is not $\boldsymbol\varepsilon'$-approx.\ optimal for~\eqref{eq:sdp},
      with
      $\boldsymbol\varepsilon' \!:=\! (\varepsilon_0,  R_Y \varepsilon_1, \varepsilon_2)$.
    \item $\|Y\|\leq R_Y$ and $\|\lambda\|\leq R_\lambda$.
  \end{itemize}
  A pair $(Y,\lambda)$ is spurious \emph{exactly} critical if the above holds for $\boldsymbol\varepsilon=0, \gamma=0$.
\end{definition}

\subsection{Statement of the theorem}

We present the main result of this section.
Let $\mathcal{A},b$ be fixed,
and let $C$ be obtained from a random perturbation of magnitude~$\sigma$ around some fixed~$\bar C$.
Concretely, $C$ is \emph{uniformly} distributed on the Frobenius ball $\ball{\bar C}{\sigma}\subset\SS^n$ of radius~$\sigma$ centered at~$\bar C$.
Consider the set $\mathscr{C}_{\boldsymbol\varepsilon,\gamma} \subset \SS^n$,
consisting of all cost matrices for which there is a spurious AFAC point:
\begin{align*}
  \mathscr{C}_{\boldsymbol\varepsilon,\gamma}
  \;\;:=\;\; \left\{ C \in \SS^n \,:
    \exists\, (Y,\lambda) \text{ an spurious $(\boldsymbol\varepsilon,\gamma)$-AFAC pair}
  \right\}.
\end{align*}
We show that if $\tau(p)>m$,
then the matrix $C$ avoids the ``bad set''
$\mathscr{C}_{\boldsymbol\varepsilon,\gamma}$
with high probability.
More precisely, the probability
$\Pr[ C \in \mathscr{C}_{\boldsymbol\varepsilon,\gamma} ]$
is vanishingly small as the ratio $\varepsilon_1/\gamma$ goes to zero.

\begin{theorem}[critical $\Rightarrow$ optimal] \label{thm:main}
  Let $p$ such that $\tau(p) {>} m$.
  Let $\boldsymbol\varepsilon \!\in\! \RR^3_+$, $\gamma \!\in\! \RR_+$.
  Let $C$ be uniformly distributed on the Frobenius ball~$\ball{\bar C}{\sigma}$.
  Then
  \begin{align*}
    \Pr[ C \in \mathscr{C}_{\boldsymbol\varepsilon,\gamma} ]
    \;\leq\;
    4 e \,  \delta^{\tau(p)-m} \,(3\kappa)^m\,(4 n^3/\sigma)^{\tau(p)},
  \end{align*}
  where $\delta := \varepsilon_1 \|\mathcal{A}\| / \gamma$
  and $\kappa := R_\lambda \|\mathcal{A}\|$,
  provided that
  $\delta < \sigma/4 n^3$.
\end{theorem}

The following corollary shows that when the stronger condition $\tau(p) > (1{+}\eta)m$ holds,
where $\eta$ is a fixed constant,
then we can derive a high probability bound while maintaining $\delta$ polynomially bounded.

\begin{corollary} \label{thm:main2}
  Consider the setup from \Cref{thm:main}.
  Assume that
  \begin{gather*}
  \tau(p) \geq (1{+}\eta)m + \eta\, t,
  \qquad
  \delta \leq (1/3\kappa)^{1/\eta}(\sigma/4 n^3)^{1{+}1/\eta},
  \end{gather*}
  for some arbitrary constants $\eta, t>0$.
  Then
  \begin{align*}
    \Pr[ C \in \mathscr{C}_{\boldsymbol\varepsilon,\gamma} ]
    \;\leq\;
    4 e \,  (\sigma/12 \kappa n^3)^{t}.
  \end{align*}
\end{corollary}
\begin{proof}
  It is a straightforward manipulation.
\end{proof}

\subsection{Tubes around varieties}

Our proof of \Cref{thm:main} relies on a geometric characterization of the set
$\mathscr{C}_{\boldsymbol\varepsilon,\gamma}$.
Such characterization is known for the case $\boldsymbol\varepsilon=0,\gamma=0$,
corresponding to \emph{exactly} critical points.
It was shown in~\cite{Boumal2018}, see also~\cite{Cifuentes2019}, that
the existence of a spurious exactly critical point implies that
$C$ lies in a certain algebraic variety of~$\SS^n$, as follows:
\begin{align*}
  \exists \text{ ({spurious} exactly critical point) }
  \quad\implies\quad
  C \,\in\, \SS^n_{n-p}+\image \mathcal{A}^*,
\end{align*}
where
$ \SS^n_{n-p} \, := \, \{ X: \rank X \leq n{-}p\} $
is a variety of bounded rank matrices,
and $\image \mathcal{A}^*$ is the linear space spanned by $A_1,\dots,A_m$.
Hence, we have that
\[\mathscr{C}_{0,0} \;\subset\; \SS^n_{n-p} + \image \mathcal{A}^*.\]
When $\tau(p) \!>\nobreak\! m$, the variety $\SS^n_{n-p} {+} \image \mathcal{A}^*$ is properly contained in~$\SS^n$.
It follows that $\mathscr{C}_{0,0}$ has measure zero
and hence, for generic~$C$, there are no spurious exactly critical points.

We show below that for approximately critical points the situation is analogous,
except that we need to consider a \emph{tubular} neighborhood around the variety~$\SS^n_{n-p}\!+\!\image\mathcal{A}^*$.

\begin{definition}
  For a set $\mathcal{W}\subset \SS^n$ and a positive number $\delta\in \RR_+$,
  we define
  \begin{align*}
    \tube_\delta\mathcal{W}
    \,:=\, \{ X \in \SS^n: \|X \!-\! W \| \leq \delta \text{ for some } W\in \mathcal{W} \}.
  \end{align*}
\end{definition}

\begin{proposition}\label{thm:sufficient}
  Let
  $\delta := {\varepsilon_1 \|\mathcal{A}\|}/{\gamma}$
  and
  $B_\lambda := \{\lambda\!\in\!\RR^m: \|\lambda\| \!\leq\! R_\lambda \}$.
  Then
  \begin{gather*}
    \mathscr{C}_{\boldsymbol\varepsilon,\gamma}
    \;\subset\; \tube_\delta(\SS^n_{n-p}) + \mathcal{A}^*(B_\lambda)
    \;=\; \tube_\delta(\SS^n_{n-p} \!+\! \mathcal{A}^*(B_\lambda) ).
  \end{gather*}
\end{proposition}

The next lemma is the key ingredient for \Cref{thm:sufficient}.
It shows that if $Y$ is a spurious AFAC point, then its smallest singular value $\sigma_p(Y)$ is bounded from below.
This is a generalization of a previous result by Burer and Monteiro~\cite{Burer2005} about exactly critical points.
An analogue of this lemma is found in \cite[Lem.3.2]{Pumir2018} for the case of EFAC points.

\begin{lemma} \label{thm:firstpart}
  Let $Y$ be an $(\boldsymbol\varepsilon,\gamma)$-AFAC point of~\eqref{eq:bm}.
  If $\sigma_p(Y) \leq {\gamma}/{\|\mathcal{A}\|}$,
  then $Y Y^{T}$ is $\boldsymbol\varepsilon'$-approximately optimal for~\eqref{eq:sdp},
  with $\boldsymbol\varepsilon' := (\varepsilon_0,  R_Y \varepsilon_1, \varepsilon_2)$.
\end{lemma}
\begin{proof}
  Let $(Y,\lambda)$ satisfy~\eqref{eq:approxcritical}, and let us show that $(Y Y^T, \lambda)$ satisfies~\eqref{eq:approxoptimal}.
  The first three conditions in~\eqref{eq:approxoptimal} are easy to check.
  We proceed to show the last one: $S(\lambda) \!\succeq\! -\varepsilon_2 \id_n$.
  Given a unit vector $u \!\in\! \RR^n$, we need to show that $u^T S(\lambda) u \!\geq\! -\varepsilon_2$.
  There is a unit vector $z\!\in\! \RR^p$ such that $\|Y z\| = \sigma_p(Y)$.
  The matrix $U \!:=\! u z^T$ satisfies $\|U\|\!=\!1$ and
  \begin{align*}
    \|U Y^T\| \leq \|u\| \|Yz\| = \sigma_p(Y) \leq \gamma/\|\mathcal{A}\|.
  \end{align*}
  Then $\|\mathcal{A}(U Y^T)\| \!\leq\! \gamma$,
  so by~\eqref{eq:conditionsecond} we have
  \[
    -\varepsilon_2
    \,\leq\, S(\lambda)\!\bullet\! U U^T
    \,=\, \|z\|^2 (u^T S(\lambda) u)
    \,=\, u^T S(\lambda) u.
    \qedhere
  \]
\end{proof}

\begin{proof}[Proof of \Cref{thm:sufficient}]
  Let $C \in \mathscr{C}_{\boldsymbol\varepsilon,\gamma}$,
  so there is a spurious $(\boldsymbol\varepsilon,\gamma)$-AFAC pair $(Y,\lambda)$.
  By \Cref{thm:firstpart}, we must have $\sigma_p(Y)\!>\!\gamma/\|\mathcal{A}\|$.
  Let $S := S(\lambda)$,
  and note that $\|S\, Y\| \!\leq\! \varepsilon_1 \!=\! \gamma\, \delta/\|\mathcal{A}\| $ by~\eqref{eq:conditionfirst}.
  Then
  \begin{align}\label{eq:distS}
    \dist(S,\,\SS^n_{n-p})
    \,=\, \sqrt{ \textstyle\sum_{i=1}^p \sigma^2_{n-i+1}(S) }
    \,\leq\, \|S\, Y\| / \sigma_p(Y)
    \,<\, \delta,
  \end{align}
  and hence
  $
    C \,=\,
    S(\lambda) \!+\! \mathcal{A}^*(\lambda)
    \,\in\, \tube_\delta( \SS^n_{n-p}) \!+\! \mathcal{A}^*(B_\lambda).
  $
\end{proof}

By \Cref{thm:sufficient},
the set of ``bad'' cost matrices $\mathscr{C}_{\boldsymbol\varepsilon,\gamma}$ is contained in a tube around a variety.
To prove \Cref{thm:main}, we need an upper bound on the probability mass of this tube.
The computation of integrals over tubes has a long history in differential geometry~\cite{Gray2012}.
Effective bounds for these integrals were shown in~\cite{Burgisser2008,Lotz2015,Basu2021};
they were used for smooth analysis in the first reference.
We will use the following recent bound.

\begin{theorem}[{\cite[Thm.1.1]{Basu2021}}] \label{thm:tube}
  Let $V\subset\RR^k$ be a real variety of codimension~$c$
  defined by polynomials of degree at most~$D$.
  Let $x$ be uniformly distributed on the Euclidean ball $\ball{\bar x}{\sigma}\subset \RR^k$.
  Then 
  \begin{align}\label{eq:tube}
    \Pr[ x \in \tube_\delta(V) ]
    \,\leq\,
    4 \left(\frac{4 k D \delta}{\sigma}\right)^{\!c\!}
    \left(1 + \frac{(4D{+}1)\delta}{\sigma}\right)^{\!k-c\!}
    \,\leq\, {4 e} \left(\frac{4 k D \delta}{\sigma}\right)^{\!c\!},
  \end{align}
  where in the second inequality we require that
  $\sigma/\delta \geq (4D{+}1)(k{-}c)$.
\end{theorem}

\begin{proof}[Proof of \Cref{thm:main}]
  Let $B_{\lambda}\subset\RR^m$ be the of radius~$R_\lambda$ centered at zero.
  Consider an $\epsilon$-net $\mathcal{N}$ of~$B_{\lambda}$,
  where $\epsilon := \delta/\|\mathcal{A}\|$.
  It is known that $(3 R_\lambda/\epsilon)^m = (3 \kappa/\delta)^m$ points suffices for the $\epsilon$-net.
  Observe that
  \begin{subequations}\label{eq:ballunion}
  \begin{gather}
    \mathcal{A}^*(B_{\lambda})
    \,\subset\,
    \mathcal{A}^*(\tube_\epsilon(\mathcal{N}))
    \,\subset\,
    \tube_\delta(\mathcal{A}^*(\mathcal{N})),
    \\
    \mathscr{C}_{\boldsymbol\varepsilon,\gamma}
    \,\subset \,
    \tube_\delta(\SS^n_{n-p}) \!+\! \mathcal{A}^*(B_{\lambda})
    \,\subset\,
    \tube_{2\delta}(\SS^n_{n-p}) \!+\!
    \mathcal{A}^*(\mathcal{N}).
  \end{gather}
  \end{subequations}
  Recall that $\SS^n_{n-p}$ is a variety of codimension~$\tau(p)$ defined by equations of degree $n{-}p{+}1$.
  For any $\ell\in \mathcal{N}$,
  \Cref{thm:tube} gives
  \begin{align*}
    \Pr[ C{-}\mathcal{A}^*(\ell) \in \tube_{2\delta}(\SS^n_{n\!-\!p}) ]
    \,<\,
    4 e \!\left(\!\frac{8 \tau(n) (n{-}p{+}1) \delta}{\sigma}\!\right)^{\!\!\tau(p)\!}
    \,<\,
    4 e \!\left(\!\frac{4 n^3 \delta}{\sigma}\!\right)^{\!\!\tau(p)\!}.
  \end{align*}
  Finally, the union bound gives
  \begin{align*}
    \Pr[ C \in \mathscr{C}_{\boldsymbol\varepsilon,\gamma} ]
    &\,\leq\,
    \#\mathcal{N} \cdot
    \Pr[ C \in \tube_{2\delta}(\SS^n_{n\!-\!p}){+}\mathcal{A}^*(\ell) ]
    \\
    &\,<\,
    \left(3\kappa/\delta\right)^m \cdot 4 e \left(4 n^3 \delta/ \sigma\right)^{\tau(p)}.
    \qedhere
  \end{align*}
\end{proof}

\section{Feasibility of critical points}\label{s:feasibility}

In this section we address the computation of a feasible solution for~\eqref{eq:sdp}.
To do so, we consider the least squares problem~\eqref{eq:ls}
and its Burer-Monteiro relaxation~\eqref{eq:ls-bm}.
Note that if~\eqref{eq:sdp} is feasible, then the optimal value of~\eqref{eq:ls} is zero.
Nevertheless, the results of this section apply to an arbitrary instance of~\eqref{eq:ls},
even if the optimal value is nonzero.

We consider a smoothed analysis setting in which $b$ is fixed,
and~$\mathcal{A}$ is subject to a small random perturbation.
We will show that problem~\eqref{eq:ls-bm} has no \emph{spurious} approximately critical (AC) points with high probability.
This means that any AC point of~\eqref{eq:ls-bm} is approximately optimal for~\eqref{eq:ls}.
We will restrict our attention to AC points~$Y$ of bounded norm.
Hence, we assume that $\|Y\| \!\leq\! R_Y$ for some fixed constant~$R_Y$.

\subsection{Spurious approximately critical points}
We proceed to describe the optimality conditions for~\eqref{eq:ls} and~\eqref{eq:ls-bm}.
Let $f(X) := \|\mathcal{A}(X){-}b\|^2$ be the least squares objective function.
The optimality conditions for the convex problem~\eqref{eq:ls} are:
\begin{align*}
  S(X) \, X = 0,
  \qquad
  X \succeq 0,
  \qquad
  S(X) \succeq 0,
\end{align*}
where $S(X)$ is the gradient of the objective function
\begin{align*}
  S(X) \,:=\, \nabla f(X)
  \,=\, 2 \mathcal{A}^*(\mathcal{A}(X) \!-\! b) \,\in\, \SS^n.
\end{align*}
We call $X$ approximately optimal if either
$f(X)$ is very close to zero,
or the above conditions are approximately satisfied.

\begin{definition}
  Let $\boldsymbol\varepsilon\!=\!(\varepsilon_0,\varepsilon_1,\varepsilon_2)\!\in\! \RR_+^3$.
  We say that $X\!\in\!\SS^n$ is $\boldsymbol\varepsilon$-\emph{approximately optimal} for~\eqref{eq:ls} if $X\succeq 0$ and either
  \begin{gather} \label{eq:ls-approxoptimal}
    \|\mathcal{A}(X){-}b\| \leq \varepsilon_0
    \quad\text{ or }\quad
    \left(\,\|S(X)  X\| \leq \varepsilon_1
      \,\text{ and }\,
    S(X) \succeq -\varepsilon_2\, \id_n\,\right).
  \end{gather}
\end{definition}

The following lemma bounds the optimality gap for the second case in~\eqref{eq:ls-approxoptimal}.

\begin{lemma} \label{thm:ls-approxoptimal}
  Let $\bar X\!\in\!\SS_+^n$ such that
  $\|S(\bar X) \bar X\| \leq \varepsilon_1$,
  $S(\bar X) \succeq -\varepsilon_2 \id_n$.
  Then
  \begin{align*}
    f(\bar X) \,\leq\, f(X) + \varepsilon_1\sqrt{n} + \varepsilon_2 \|X\| \sqrt{n}
    \qquad\forall\; X \in \SS^n_+.
  \end{align*}
\end{lemma}
\begin{proof}
  Let $L(X) := f(X) - \bar S\bullet X$, with $\bar S:= S(\bar X)$.
  This is a convex function with $\nabla L(\bar X)=0$,
  so $\bar X$ is its global minimum.
  Note that
  \begin{gather*}
    f(X)
    \,=\, L(X) + \bar S\bullet X
    \,\geq\, L(X) - (\varepsilon_2 \id_n)\bullet X
    \,=\, L(X) - \varepsilon_2 \|X\| \sqrt{n},
    \\
    L(\bar X)
    \,=\, f(\bar X) - \bar S \bullet \bar X
    \,\geq\, f(\bar X) - \|\bar S \bar X\|_*
    \,\geq\, f(\bar X) - \varepsilon_1 \sqrt{n}.
  \end{gather*}
  Since $L(X) \geq L(\bar X)$, the result follows from the above equations.
\end{proof}

We proceed now to the Burer-Monteiro problem~\eqref{eq:ls-bm}.
This is a special instance of the unconstrained optimization problem~\eqref{eq:unconstrained}.
The criticality conditions for~\eqref{eq:ls-bm} are obtained by specializing~\eqref{eq:conditionsunconstrainedapprox}.

\begin{definition}
Let $(\varepsilon_1,\varepsilon_2) \!\in\! \RR_+^2$.
We say that $Y\!\in\!\RR^{n\times p}$ is $(\varepsilon_1,\varepsilon_2)$-\emph{approximately 2-critical} (AC) for~\eqref{eq:ls-bm} if
\begin{subequations}\label{eq:ls-approxcritical}
\begin{gather}
  \|S(Y Y^T)  Y \| \leq \varepsilon_1,
  \\
  \label{eq:ls-conditionsecond}
  S(Y Y^T) \bullet U U^T + 4 \|\mathcal{A}(U Y^T)\|^2
  \!\geq\! -\varepsilon_2,
  \quad\forall\, U\!\in\! \RR^{n\times p},\, \|U\|\!=\!1
\end{gather}
\end{subequations}
\end{definition}

We are ready to formalize the concept of spurious critical points.

\begin{definition}
  Let $R_Y$ be fixed.
  For an arbitrary $\boldsymbol\varepsilon \!=\! (\varepsilon_0,\varepsilon_1,\varepsilon_2) \!\in\! \RR_+^3$,
  we say that a point $Y$ is \emph{spurious} $\boldsymbol\varepsilon$-AC if:
  \begin{itemize}
    \item
      $Y$ is $(\varepsilon_1,\varepsilon_2)$-AC for~\eqref{eq:ls-bm}.
    \item
      $Y Y^T$ is not $\boldsymbol\varepsilon'$-approx.\ optimal for~\eqref{eq:ls},
      with
      $\boldsymbol\varepsilon' := (\varepsilon_0,R_Y \varepsilon_1, 5\varepsilon_2)$.
    \item $\|Y\|\leq R_Y$.
  \end{itemize}
  A pair $(Y,\lambda)$ is spurious \emph{exactly} critical if the above holds for $\boldsymbol\varepsilon=0$.
\end{definition}

\subsection{Statement of the theorem}

We assume here that $b\in\RR^m$ is fixed, and we vary the linear map~$\mathcal{A}$.
We identify $\mathcal{A}$ with the tuple of matrices $(A_1,\dots,A_m)$,
and hence view it as an element of the Euclidean space~$(\SS^n)^m$.
We assume that $\mathcal{A}$ is \emph{uniformly} distributed on the Euclidean ball $\ball{\bar{\mathcal{A}}}{\sigma}\subset (\SS^n)^m$ of radius~$\sigma$ centered at~$\bar{\mathcal{A}}$.
Consider the set $\mathscr{A}_{\boldsymbol\varepsilon} \subset (\SS^n)^m$,
consisting of all~$\mathcal{A}$ for which there is a spurious AC point:
\begin{align*}
  \mathscr{A}_{\boldsymbol\varepsilon}
  \;\;:=\;\; \left\{ \mathcal{A} \in (\SS^n)^m \,:
      \exists\, Y \text{ a spurious $\boldsymbol\varepsilon$-AC point}
  \right\}.
\end{align*}
We show that if $\tau(p)>m$,
then the probability
$\Pr[ \mathcal{A} \in \mathscr{A}_{\boldsymbol\varepsilon} ]$
is vanishingly small as the ratio $\varepsilon_1/\sqrt{\varepsilon_2}$ goes to zero.

\begin{theorem}[critical $\Rightarrow$ feasible] \label{thm:ls-main}
  Let $p$ such that $\tau(p) {>} m$,
  and $\boldsymbol\varepsilon \!\in\! \RR^3_+$.
  Let $\mathcal{A}$ be uniformly distributed on the Euclidean ball~$\ball{\bar{\mathcal{A}}}{\sigma}$.
  Then
  \begin{align*}
    \Pr[\mathcal{A}\in \mathscr{A}_{\boldsymbol\varepsilon}]
    \;\leq\;
    4 e \,  \delta^{\tau(p)-m} \,(3\kappa)^m\,(2 n^3 m/\sigma \varepsilon_0)^{\tau(p)},
  \end{align*}
  where $\delta := \varepsilon_1 R_A / \sqrt{\varepsilon_2}$,
  $\kappa := 2(R_A R_Y^2 \!+\! \|b\|)R_A$,
  and $R_A := \|\bar{\mathcal{A}}\|\!+\!\sigma$,
  provided that
  $\delta < \sigma\varepsilon_0/4 n^3 m$.
\end{theorem}

As before, we can derive a high probability bound with $\delta$ polynomially bounded when $\tau(p) > (1{+}\eta)m$.

\begin{corollary} \label{thm:ls-main2}
  Consider the setup from \Cref{thm:ls-main}.
  Assume that
  \begin{gather*}
  \tau(p) \geq (1{+}\eta)m + \eta\, t,
  \quad
  \delta \leq (1/3\kappa)^{1/\eta}(\rho\sigma^2/2 n^3 m)^{1{+}1/\eta},
  \quad
  \varepsilon_0 \geq \rho\, \sigma,
  \end{gather*}
  for some arbitrary constants $\eta, t, \rho >0$.
  Then
  \begin{align*}
    \Pr[ \mathcal{A} \in \mathscr{A}_{\boldsymbol\varepsilon} ]
    \;\leq\;
    4 e \,  (\rho\sigma^2/6 \kappa n^3 m)^{t}.
  \end{align*}
\end{corollary}

\subsection{Tubes around varieties}
As in \Cref{s:optimality}, our proof of \Cref{thm:main} relies on a geometric characterization of the spurious AC points.
First consider the simpler case of spurious \emph{exactly} critical points ($\boldsymbol\varepsilon=0$).
The following equation is a consequence of our analysis:
\begin{align*}
  \exists \text{ (spurious exactly critical point) }
  \;\;\implies\;\;
  \SS^n_{n-p} \cap \image \!\mathcal{A}^* \text{ is nontrivial}.
\end{align*}
This implies that the set $\mathscr{A}_{0}$
has measure zero when $\tau(p) \!>\! m$.
Hence, for generic~$\mathcal{A}$, there are no spurious exactly critical points.

The case $\boldsymbol\varepsilon > 0$ is similar,
but we need consider a tube around~$\SS^n_{n-p}$.

\begin{proposition}\label{thm:ls-sufficient}
  Let
  $\delta := {\varepsilon_1 R_A}/{\sqrt{\varepsilon_2}}$,
  $D_\lambda := \{\lambda\!\in\!\RR^m: 2\varepsilon_0 \!\leq\! \|\lambda\| \!\leq\! R_\lambda \}$,
  $R_\lambda := 2(R_A R_Y^2 \!+\! \|b\|)$.
  Then for $\mathcal{A} \in \mathscr{A}_{\boldsymbol\varepsilon}$ we have that
  \begin{gather*}
    \tube_\delta(\SS^n_{n-p}) \,\cap\, \mathcal{A}^*(D_\lambda)
    \,\neq\, \emptyset.
  \end{gather*}
\end{proposition}

The above proposition relies on an analogue of \Cref{thm:firstpart}.

\begin{lemma} \label{thm:ls-firstpart}
  Let $Y$ be an $(\varepsilon_1,\varepsilon_2)$-AC point of~\eqref{eq:ls-bm}.
  If $\sigma_p(Y) \leq \sqrt{\varepsilon_2}/R_A$,
  then $Y Y^{T}$ is $\boldsymbol\varepsilon'$-approximately optimal for~\eqref{eq:ls},
  with $\boldsymbol\varepsilon' := (0,R_Y \varepsilon_1, 5\varepsilon_2)$.
\end{lemma}
\begin{proof}
  Let $Y$ satisfy~\eqref{eq:ls-approxcritical}, and let us show that $Y Y^T$ satisfies~\eqref{eq:ls-approxoptimal}.
  The first-order condition is easy to check.
  We proceed to show that
  $u^T S(X) u \!\geq\! -\varepsilon_2'$
  for any unit vector $u \!\in\! \RR^n$.
  Let $z\!\in\! \RR^p$ be a unit vector such that $\|Y z\| = \sigma_p(Y) $.
  The matrix $U \!:=\! u z^T$ satisfies $\|U\|\!=\!1$ and $\|U Y^T\| \leq \|u\| \|Yz\| = \sigma_p(Y)$.
  Then $\|\mathcal{A}(U Y^T)\| \leq \sqrt{\varepsilon_2}$ and by \eqref{eq:ls-conditionsecond} we have
  \[
    u^T S(X) u
    = S(Y Y^T)\bullet U U^{T\!}
    \geq -\varepsilon_2 \!-\! 4 \|\mathcal{A}(U Y^{T\!})\|^2
    \geq -5\varepsilon_2.
    \qedhere
  \]
\end{proof}

\begin{proof}[Proof of \Cref{thm:ls-sufficient}]
  As $\mathcal{A} \in \mathscr{A}_{\boldsymbol\varepsilon}$,
  there is a spurious $\boldsymbol\varepsilon$-AC point~$Y$.
  By \Cref{thm:firstpart}, we must have $\sigma_p(Y)\!>\!\sqrt{\varepsilon_2}/R_A$.
  Note that $\|S(Y Y^T) Y\| \!\leq\! \varepsilon_1 $.
  Together with~\eqref{eq:distS},
  we conclude that $S(Y Y^T) \in \tube_\delta(\SS^n_{n-p})$.
  Let $\lambda:= 2(\mathcal{A}(Y Y^T)\!-\!b)$.
  Note that $\|\lambda\|>2\varepsilon_0$ and
  \begin{align*}
    \|\lambda\|
    = 2\, \| \mathcal{A}(Y Y^T)-b \|
    \leq 2(\|\mathcal{A}\| \|Y\|^2 + \|b\|)
    \leq R_\lambda.
  \end{align*}
  Then $\lambda \in D_\lambda$ and
  $\mathcal{A}^*(\lambda) = S(Y Y^T) \in \tube_\delta(\SS^n_{n-p})$.
\end{proof}

\begin{proof}[Proof of \Cref{thm:ls-main}]
  The result in \Cref{thm:ls-sufficient} can be expressed as:
  \begin{align*}
    \mathscr{A}_{\boldsymbol\varepsilon} \,\subset\,
    \{ \mathcal{A}\in (\SS^n)^m:
      0 \in \tube_\delta(\SS^n_{n-p}) \!+\!  \mathcal{A}^*(D_\lambda)
    \},
  \end{align*}
  which is closer to the formula in \Cref{thm:sufficient}.
  Consider an $\epsilon$-net $\mathcal{N}$ of~$D_\lambda$,
  where $\epsilon := \delta/ R_A$.
  It suffices to take
  $(3 R_\lambda/\epsilon)^m = (3 \kappa/\delta)^m$
  points for the $\epsilon$-net.
  A reasoning similar to~\eqref{eq:ballunion} gives
  \begin{align*}
    \mathscr{A}_{\boldsymbol\varepsilon}
    &\,\subset\,
    \{ \mathcal{A} \in (\SS^n)^m :
      0 \in \tube_{2\delta}(\SS^n_{n-p}) + \mathcal{A}^*(\mathcal{N}))
    \}
    \\
    &\,=\,
    \bigcup_{\ell \in \mathcal{N}}\,
    \{ \mathcal{A} \in (\SS^n)^m :
      \mathcal{A}^*(\ell)\in \tube_{2\delta}(\SS^n_{n-p}) )
    \}.
  \end{align*}
  Let $\ell \in \mathcal{N}$,
  and consider the linear map
  \begin{align*}
    \phi_\ell:(\SS^n)^m \to \SS^n, \quad
    \mathcal{A} \mapsto \mathcal{A}^*(\ell).
  \end{align*}
  This is a surjective map.
  Moreover, the scaled map $\tfrac{1}{\|\ell\|}\,\phi_\ell\,$ gives an isometry ${(\ker \phi_\ell)^\perp} \cong \SS^n$.
  It follows that
  \begin{align*}
    \phi_\ell(\mathcal{A}) \,\in\, \tube_{2\delta}(\SS^n_{n-p})
    \quad\iff\quad
    \mathcal{A} \,\in\, \tube_{2\delta/\|\ell\|}\left(\phi_\ell^{-1}(\SS^n_{n-p})\right).
  \end{align*}
  Since $\|\ell\| \geq 2\varepsilon_0$, we conclude that
  \begin{align*}
    \mathscr{A}_{\boldsymbol\varepsilon} \,\subset\, \bigcup_{\ell \in \mathcal{N}} \tube_{\delta/\varepsilon_0}(V_\ell),
    \quad\text{ with }\;
    V_\ell := \phi_\ell^{-1}(\SS^n_{n-p}).
  \end{align*}

  The final part of the proof is similar to the one in \Cref{thm:main}.
  The variety~$V_\ell$ is a cylinder over~$\SS^n_{n-p}$,
  so it has the same codimension~$\tau(p)$ and degree $n{-}p{+}1$ as~$\SS^n_{n-p}$,
  The ambient space is $(\SS^n)^m$, of dimension $\tau(n) m$.
  Using the union bound and \Cref{thm:tube}, we get
  \begin{align*}
    \Pr[\mathcal{A}\in \mathscr{A}_{\boldsymbol\varepsilon}]
    &\,<\,
    \# \mathcal{N} \cdot \Pr[\mathcal{A}\in \tube_{\delta/\varepsilon_0}(V_\ell)]
    \\
    &\,<\,
    \left({3\kappa}/{\delta}\right)^{m} \cdot
    4 e \left( {2 n^3 m \delta}/{\sigma \varepsilon_0}\right)^{\tau(p)}.
    \qedhere
  \end{align*}
\end{proof}

\section{Overall complexity estimates}\label{s:complete}

In this section we will derive polynomial time guarantees for the Burer-Monteiro method.
In particular, we will formalize and prove \Cref{thm:informalQ2a,thm:informalQ2b} from the introduction.

We introduce some notation that will be used throughout this section.
We consider two constants $\alpha \geq \beta > 0$ and associated sets $\mathfrak{M}_\alpha \supset \mathfrak{M}_\beta$,
where
$\mathfrak{M}_t := \{ Y: \|\mathcal{A}(Y Y^T){-}b\| \leq t\}$.
We assume that $\beta$ is small enough so that $\rho$-LICQ holds globally on $\mathfrak{M}_\beta$.
On the other hand, $\alpha > 0$ is sufficiently large so that a point $Y_0 \in \mathfrak{M}_\alpha$ is always known.
We further assume that $\mathfrak{M}_\alpha$ is compact.
This compactness assumption is satisfied, for instance,
if one of the constraints matrices $A_i$ is positive definite.

\subsection{Optimality}

We assume first that an approximately feasible solution~$Y_0$ is known.
Consider the following setting:
\begin{itemize}
  \item
    $p$ satisfies $\tau(p) \geq (1{+}\eta)m \!+\! \eta\, t$
    for some given constants $\eta,t \!\in\! \RR_+$.
  \item
    $\mathcal{A},b$ are fixed and $C$ is uniformly distributed on a ball $\ball{\bar C}{\sigma}$.
  \item $\exists\beta\in\RR_+$ such that:
    $\mathfrak{M}_\beta$ is compact,
    a point $Y_0\in \mathfrak{M}_\beta$ is known,
    and $\varrho$-LICQ holds on $\mathfrak{M}_\beta$.
  \item
    $R_Y, L_f \in \RR_+$ are constants that bound $\|Y\|$ and $\|C Y\|$,
    for $Y \!\in \mathfrak{M}_\beta$.
  \item
    \eqref{eq:bm} is solved with the method from \Cref{thm:complexity}
    initialized at~$Y_0$.
\end{itemize}

The next theorem shows that the Burer-Monteiro method solves~\eqref{eq:sdp} in polynomial time with high probability.

\begin{theorem}[Polytime optimality] \label{thm:complete}
  Let $\rho \!\in\! (0,1]$ arbitrary, and let
  \begin{gather*}
    \varepsilon_0 := \gamma := \epsilon,
    \quad
    \varepsilon_1 := \epsilon^2,
    \quad
    \varepsilon_2 := 16\, R_\lambda^3\, \epsilon,
    \;\;\text{ with }\;\;
    \epsilon :=
    K^{-1}\, \rho\, (\sigma/4 n^3)^{1+1/\eta},
  \end{gather*}
  where $R_\lambda$ and $K$ are the problem dependent constants
  \begin{gather*}
    R_\lambda := 2 + 2 \varrho^{-1} L_f,
    \quad
    K := \|\mathcal{A}\|\, (3 \kappa)^{1/\eta},
    \quad\text{ with }\quad
    \kappa := R_\lambda \|\mathcal{A}\|.
  \end{gather*}
  The algorithm from \Cref{thm:complexity} returns a pair $(Y,\lambda)$ after $O(\epsilon^{-6})$ function evaluations.
  With probability at least
  $1 \!-\! O(\sigma/n^3)^{t} $,
  the pair $(Y Y^T,\lambda)$ is
  $(\epsilon,\epsilon^2 R_Y,16 R_\lambda^3 \epsilon)$-approximately optimal for~\eqref{eq:sdp}.
\end{theorem}

\begin{proof}
  The smoothness assumptions in \Cref{thm:complexity} are satisfied
  since $\mathfrak{M}_\beta$ is compact.
  Then $(Y,\lambda)$ is an $(\boldsymbol\varepsilon,\gamma)$-AFAC pair with $\|\lambda\| \!\leq\! R_\lambda$.
  Note that
  \begin{align*}
    \delta
    \,:=\,
    \varepsilon_1 \|\mathcal{A}\| / \gamma
    \,=\,
    \epsilon\, \|\mathcal{A}\|\,
    \,\leq\,
    (1/3 \kappa)^{1/\eta}\,(\sigma/2 e n^3)^{1+1/\eta}
  \end{align*}
  is as in \Cref{thm:main2}.
  Hence $(Y Y^T, \lambda)$ is $(\varepsilon_0,\varepsilon_1 R_Y,\varepsilon_2)$-approximately optimal for~\eqref{eq:sdp}
  with probability $1 \!-\! O(\sigma/n^3)^{t}.$
\end{proof}

The above theorem shows that $Y Y^T$ obtained is approximately optimal for the perturbed problem~\eqref{eq:sdp} with high probability.
Let \eqsdp denote the SDP problem in which we use the unperturbed cost matrix~$\bar{C}$.
We can also show that $Y Y^T$ is also approximately optimal for~\eqsdp.

\begin{corollary} \label{thm:complete2}
  Consider the setup of \Cref{thm:complete}.
  With probability at least
  $1 \!-\! O(\sigma/n^3)^{t} $,
  the pair $(Y Y^T,\lambda)$ is
  $(\varepsilon''_0,\varepsilon''_1,\varepsilon''_2)$-approximately optimal for~\eqsdp,
  where $\varepsilon''_0,\varepsilon''_1,\varepsilon''_2 = O(\sigma)$.
\end{corollary}
\begin{proof}
  Let $X := Y Y^T$.
  We know that $(X, \lambda)$ is $(\varepsilon'_0,\varepsilon'_1,\varepsilon'_2)$-approximately optimal for~\eqref{eq:sdp} with high probability.
  Let $S := C \!-\! \mathcal{A}^*(\lambda)$, $\bar S := \bar C \!-\! \mathcal{A}^*(\lambda)$
  be the slack matrices for \eqref{eq:sdp} and \eqsdp.
  Observe that
  \begin{gather}
    \|\mathcal{A}(X){-}b\|
    \,\leq\,
    \varepsilon_0'
    \,\leq\,
    O(\sigma),
    \\
    \label{eq:sbar}
    \| \bar S X \|
    \,\leq\,
    \| S X \| + \| (\bar S {-} S) X \|
    \,\leq\,
    \varepsilon_1' + \sigma \|X\|
    \,\leq\,
    O(\sigma),
    \\
    \label{eq:sbar2}
    \bar S
    \,\succeq\,
    S
    - \|\bar S {-} S\| \id_n
    \succeq
    -(\varepsilon_2' \!+\! \sigma)\id_n
    \succeq
    -O(\sigma)\id_n.
  \end{gather}
  So the optimality conditions of \eqsdp hold with
  $\varepsilon''_0,\varepsilon''_1,\varepsilon''_2 = O(\sigma)$.
\end{proof}

\subsection{Feasibility}

We now consider the computation of a feasible point.
Consider the following setting:
\begin{itemize}
  \item
    $p$ satisfies $\tau(p) \geq (1{+}\eta)m \!+\! \eta\, t$
    for some given constants $\eta,t \!\in\! \RR_+$.
  \item
    $b$ is fixed and $\mathcal{A}$ is uniformly distributed on a ball
    $\ball{\bar{\mathcal{A}}}{\sigma}$.
  \item $\exists\alpha\in\RR_+$ and a matrix $Y_0$ such that
    $\mathfrak{M}_{\alpha}$ is compact and $Y_0 \in \mathfrak{M}_{\alpha}$
  \item
    $R_Y \in \RR_+$ is a constant that bounds $\|Y\|$,
    for $Y \!\in \mathfrak{M}_\alpha$.
  \item
    \eqref{eq:ls-bm} is solved with the method from \Cref{thm:complexityunconstrained}
    initialized at~$Y_0$.
\end{itemize}

\begin{theorem}[Polytime feasibility] \label{thm:ls-complete}
  Let $\rho \!\in\! (0,1]$ arbitrary,
  and let
  \begin{gather*}
    \varepsilon_1 \,:=\, \epsilon^{3/2},
    \quad
    \varepsilon_2 \,:=\, \epsilon,
    \quad\text{ with }\quad
    \epsilon \,:=\,
    K^{-1}\, \left(\rho\, \sigma^2/2 n^3 m\right)^{1+1/\eta},
  \end{gather*}
  where $K$ is the problem dependent constant
  \begin{gather*}
    K := R_A\,(3\kappa)^{1/\eta},
    \quad\text{ with }\quad
    \kappa := 2(R_A R_Y^2 \!+\! \|b\|)R_A,
    \quad
    R_A := \|\bar{\mathcal{A}}\| \!+\! \sigma.
  \end{gather*}
  The algorithm from \Cref{thm:complexityunconstrained} returns a point $Y$ after $O(\epsilon^{-3})$ function evaluations.
  With probability at least
  $1 {-} O(\sigma^2/n^3 m)^{t} $,
  we have that $Y Y^T$ is
  $( \rho\, \sigma, \epsilon^{3/2} R_Y, 5 \epsilon)$-approximately optimal for~\eqref{eq:ls}.
\end{theorem}

\begin{proof}
  The smoothness assumptions in \Cref{thm:complexityunconstrained} are satisfied since $\mathfrak{M}_\alpha$ is compact.
  Therefore $Y$ is an $(\varepsilon_1,\varepsilon_2)$-AC point.
  Note that
  \begin{align*}
    \delta \,:=\,
    \varepsilon_1 R_A/ \sqrt{\varepsilon_2}
    \,=\,
    \epsilon\, R_A
    \,=\,
    (1/3\kappa)^{1/\eta} (\rho\, \sigma^2/2 n^3 m)^{1+1/\eta}
  \end{align*}
  is as in \Cref{thm:ls-main2}.
  Hence $Y Y^T$ is $( \rho\, \sigma, \varepsilon_1 R_Y, 5 \varepsilon_2)$-approximately optimal for~\eqref{eq:ls}
  with probability $1 \!-\! O(\sigma^2/n^3 m)^{t}. $
\end{proof}

\begin{remark}
  The above theorem holds even if the optimal value of~\eqref{eq:ls} is nonzero.
  In the special case that the optimal value is zero,
  then by \Cref{thm:ls-approxoptimal} we have that
  \begin{align*}
    \|\mathcal{A}(Y Y^T) {-} b\| \leq
    \max\{ \varepsilon_0',\;
    {n}^{1/4} (\varepsilon_1'{+}\varepsilon_2' R_Y)^{1/2}\}.
  \end{align*}
\end{remark}

Let \eqls denote the instance of problem \eqref{eq:ls} in which we use the unperturbed constraints~$\bar{\mathcal{A}}$.
We next show that $Y Y^T$ is also approximately optimal for~\eqls.

\begin{corollary} \label{thm:ls-complete2}
  Consider the setup of \Cref{thm:ls-complete}.
  With probability at least
  $1 {-} O(\sigma^2/n^3 m)^{t} $,
  the matrix $Y Y^T$ is
  $(\varepsilon_0'',\varepsilon_1'',\varepsilon_2'')$-approximately optimal for~\eqls, where
  $\varepsilon_0'',\varepsilon_1'',\varepsilon_2'' = O(\sigma)$.
\end{corollary}
\begin{proof}
  We know that the matrix $X := Y Y^T$ is $(\varepsilon'_0,\varepsilon'_1,\varepsilon'_2)$-approximately optimal for~\eqref{eq:sdp} with high probability.
  There are two cases.
  The first case is that
  $\|{\mathcal{A}}(X) {-} b\| \leq\nobreak \varepsilon_0'$,
  which implies that
  \begin{align*}
    \|\bar{\mathcal{A}}(X) {-} b\|
    \,\leq\,
    \|{\mathcal{A}}(X) {-} b\|
    + \|(\bar{\mathcal{A}} {-} \mathcal{A})X \|
    \,\leq\,
    \varepsilon_0' + \sigma \|X\|
    \,\leq\,
    O(\sigma).
  \end{align*}
  This means that $\varepsilon_0'' = O(\sigma)$.
  Consider the following variables:
  \begin{align*}
    \lambda := 2(\mathcal{A}(X) \!-\! b),\quad
    S := \mathcal{A}^*(\lambda),\quad
    \bar\lambda := 2(\bar{\mathcal{A}}(X) \!-\! b),\quad
    \bar S := \bar{\mathcal{A}}^*(\bar \lambda).
  \end{align*}
  The second case is that
  $\|S X\| \leq \varepsilon_1$ and $S \succeq -\varepsilon_2 \id_n$.
  Observe that
  \begin{gather*}
    \|\bar \lambda {-} \lambda \|
    \,\leq\,
    2\, \|\bar{\mathcal{A}}{-}\mathcal{A}\| \, \|X\|
    \,\leq\,
    O(\sigma),
    \\
    \|\bar S {-} S \|
    \,\leq\,
    \|\bar{\mathcal{A}}^*(\bar\lambda {-} \lambda)\|
    + \|(\bar{\mathcal{A}}^*{-} {\mathcal{A}}^*)\lambda\|
    \,\leq\,
    O(\sigma).
  \end{gather*}
  From~\eqref{eq:sbar} and~\eqref{eq:sbar2} we get that
  $\|\bar S X\| \!\leq\! O(\sigma)$ and $\bar S \!\succeq\! -O(\sigma) \id_n$.
  So the optimality conditions of \eqls hold with
  $\varepsilon''_0,\varepsilon''_1,\varepsilon''_2 = O(\sigma)$.
\end{proof}

\section{Experiments}\label{s:experiments}

We present some experimental results to complement our theorems.
We rely on the open source library \texttt{NLopt}~\cite{nlopt}
to solve the nonlinear problem~\eqref{eq:bm}.
Concretely, we use the \emph{augmented Lagrangian method} (ALM) implemented in \texttt{NLopt}
(which is based on~\cite{Birgin2008}),
and we use the preconditioned truncated Newton method as the subroutine.
We also rely on \texttt{Mosek} to solve semidefinite programs.

For our first experiment we consider a random SDP with a planted solution.
More precisely, we consider a matrix $X_0 \in \SS^n$, $X_0 \succeq 0$ of rank~$r$,
where $n:= 50$ and $r \in \{4,7,12\}$.
We then generate a random SDP for which $X_0$ is an optimal solution.
To do so, we first generate $m := \tau(r)$ random constraints that are satisfied at $X_0$.
Afterwards, we find a cost matrix $C$ in the normal cone of~$X_0$ (this requires solving an auxiliary SDP).

For each $r\in \{4,7,12\}$ we generate 100 random SDPs as above.
We solve these SDPs with the Burer-Monteiro method,
using different values of~$p$ (the rank of~$Y$) and random initializations.
The initial points are matrices with i.i.d.\ normalized Gaussian entries.
\Cref{fig:exp1} shows the percentage of experiments solved correctly for each value of $r$ and~$p$.
We regard an experiment as ``correct'' if the criticality conditions from~\eqref{eq:approxoptimal} are satisfied.

\begin{figure}[htb]
  \centering
  \includegraphics[height=130pt]{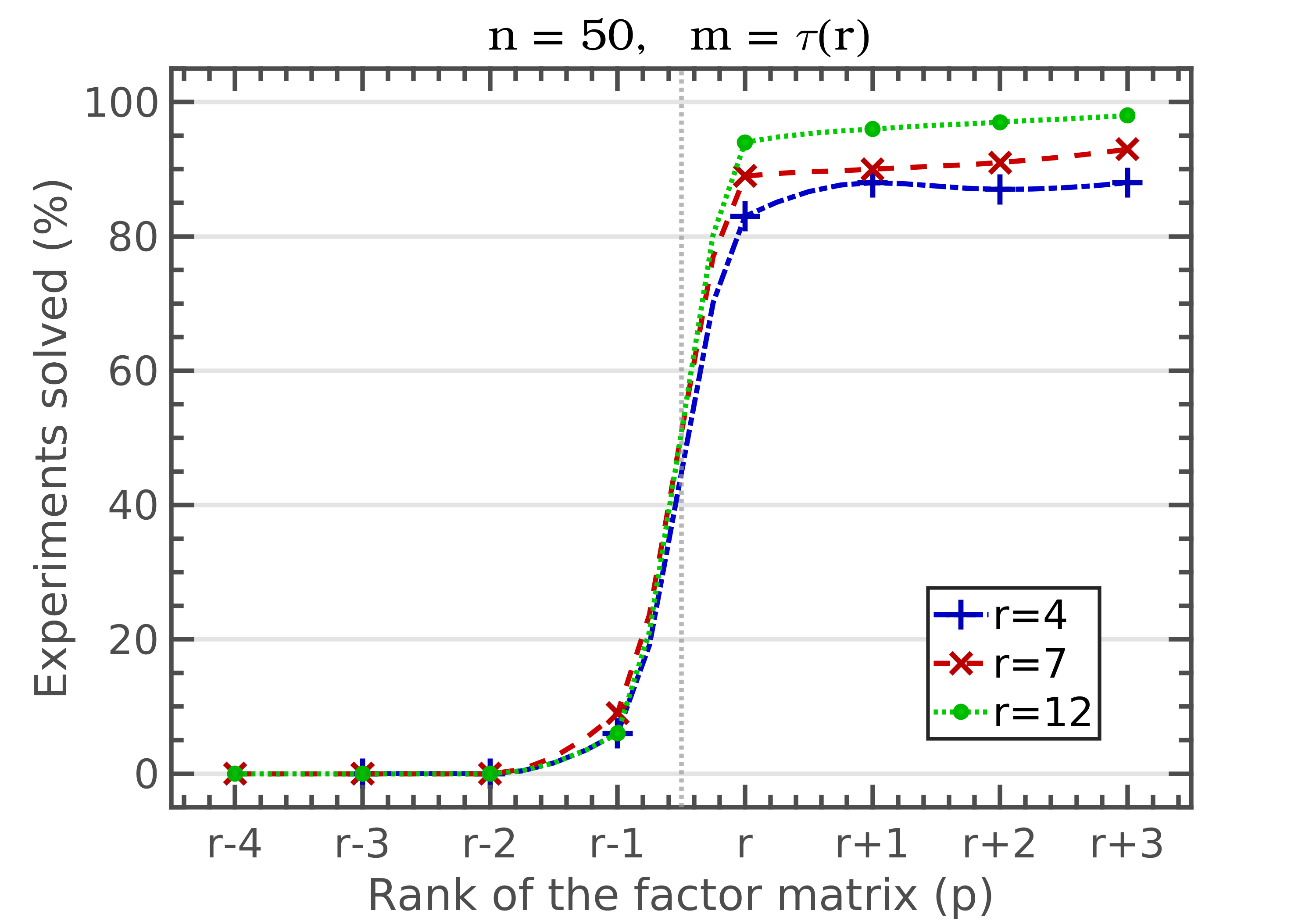}
  \caption{
    Performance of the Burer-Monteiro method for a class of random SDPs with a planted solution of rank~$r$.
  }
  \label{fig:exp1}
\end{figure}

\Cref{fig:exp1} illustrates that there is a sharp phase transition at the Barvinok-Pataki bound $p=\nobreak r$.
Above the Barvinok-Pataki bound, the Burer-Monteiro method solves most instances.
Beneath the Barvinok-Pataki bound, it is not just that our techniques stop working, but that the method itself usually fails.
Note that previous work~\cite{Pumir2018,Bhojanapalli2018} required $p$ to be larger than~$3 r/\sqrt{2}$.
However, we see from experiments that the phase transition is much sharper than this.
It remains an interesting open question to investigate for what types of structured SDPs having a smaller $p$ might suffice \cite{Bandeira2016,Mei2017}.

Observe that, even for $p\geq r$, the number of experiments solved is always below~$100\%$.
Nonetheless, the number of bad instances seems to get smaller for larger values of~$p$.
This agrees with our result from \Cref{thm:main}.

For the second experiment we fix the parameters $n:=50$, $m:=28$, and $p:=r:=7$.
Among the 100 random SDPs considered in \Cref{fig:exp1},
we take an instance for which the Burer-Monteiro problem performed badly.
We then perturb this seemingly bad SDP by adding varying amounts of noise~$\sigma$.
For each noise level we solve 70 random experiments,
in which both the perturbations and the initializations are random.
The perturbations consist in adding to each matrix a random matrix with i.i.d.\ Gaussian entries scaled by the noise level.
\Cref{fig:exp2} summarizes the results obtained.

\begin{figure}[htb]
  \centering
  \begin{subfigure}[b]{0.4\textwidth}
    \includegraphics[height=130pt]{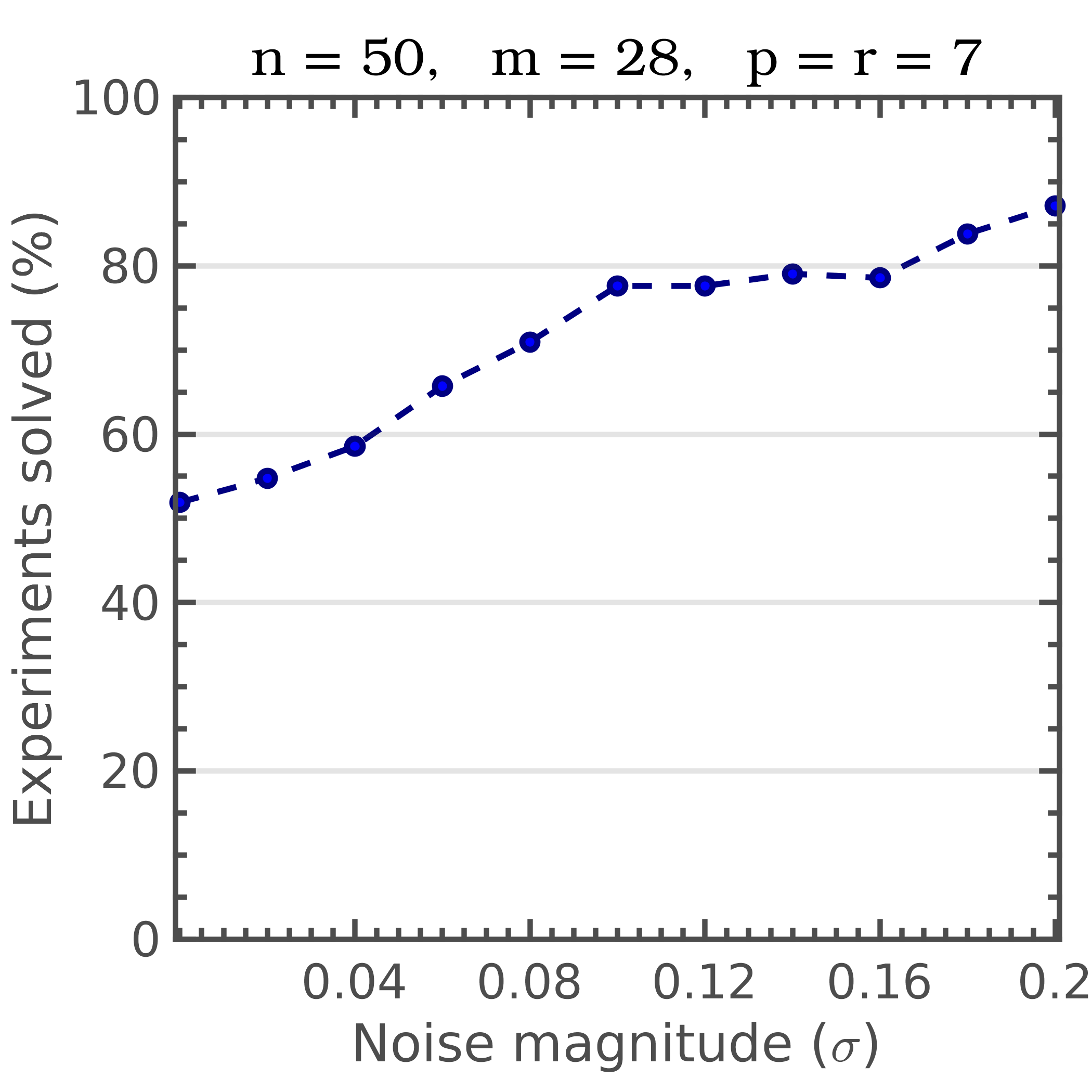}
    \caption{Fraction of SDPs solved}
    \label{fig:exp2a}
  \end{subfigure}
  \begin{subfigure}[b]{0.55\textwidth}
    \includegraphics[height=130pt]{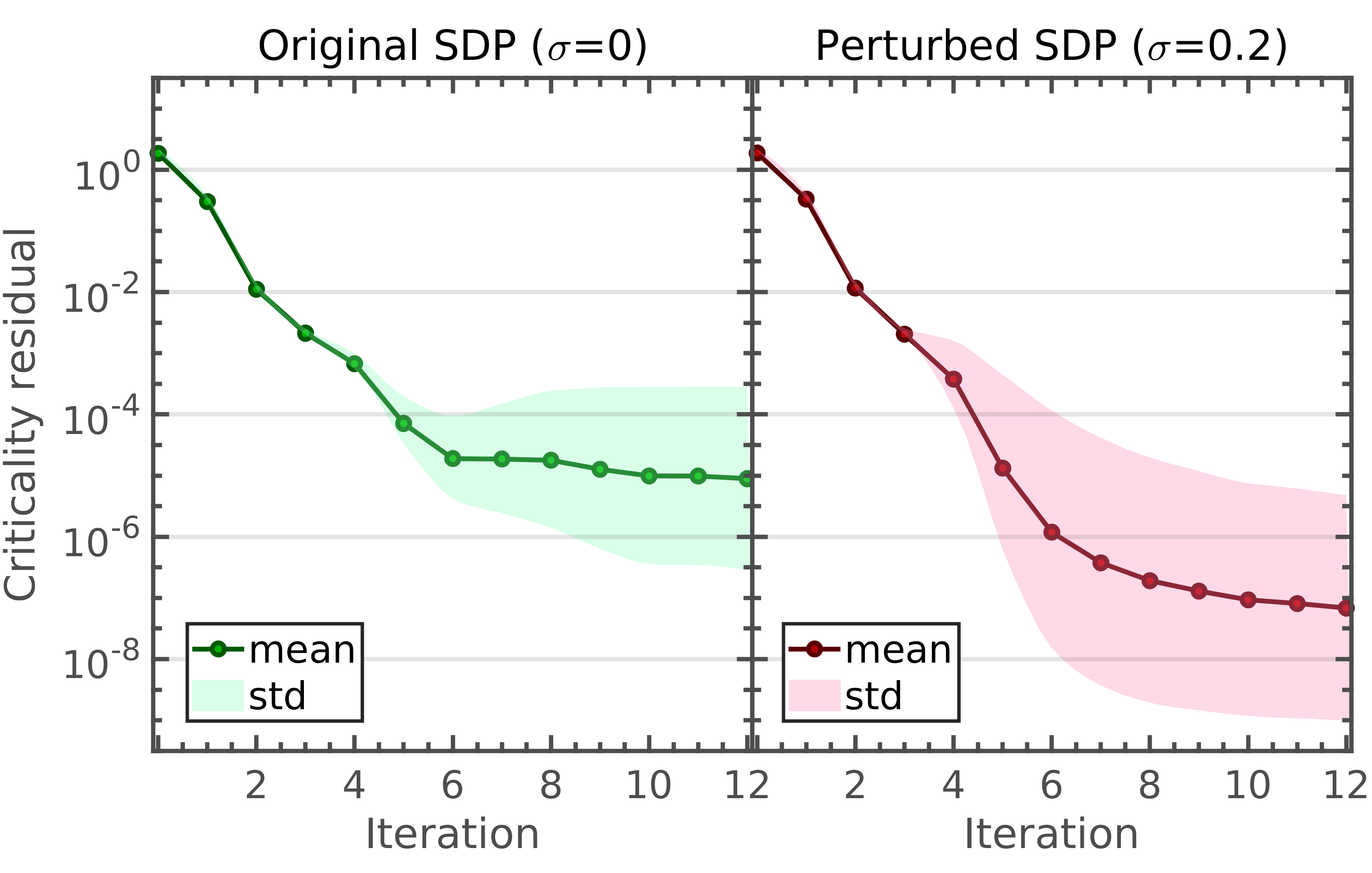}
    \caption{Iteration progress}
    \label{fig:exp2b}
  \end{subfigure}
  \caption{
    Performance of the Burer-Monteiro method
    after perturbing a ``bad'' SDP with different levels of noise.
  }
  \label{fig:exp2}
\end{figure}

\Cref{fig:exp2a} shows the percentage of instances for which the Burer-Monteiro method succeeded for each noise level.
The percentage is with respect to the random perturbation and the random initialization.
For the unperturbed problem ($\sigma\!=\!0$) the method succeeds only for $52\%$ of the random initializations.
This percentage increases as we add noise.
For $\sigma\!=\!0.2$ the method succeeds $87\%$ of the time.
\Cref{fig:exp2b} shows the progress of the algorithm for the cases $\sigma\!=\!0$ and $\sigma\!=\!0.2$.
The progress is measured in terms of the residual of the criticality conditions.
The figure shows the mean and standard deviation of the residual for each iteration of ALM.

\Cref{fig:exp2a,fig:exp2b} illustrate the advantages of smoothing a badly behaved SDP.
\Cref{thm:complete} predicts that the complexity of the algorithm is proportional to $\sigma^{-d}$ for some  exponent~$d$.
So for a fixed number of iterations~$N$,
we should set the noise level proportional to $N^{-1/d}$.
However, our bounds were shown for the algorithm from \Cref{thm:complexity}.
We do not know if they also apply to~ALM.

\section{Discussion}

\subsubsection*{Assumptions of our theorems}

We proceed to discuss the main assumptions of \Cref{thm:complete,thm:ls-complete}.
A major premise in both theorems is that $\tau(p) > (1{+}\eta)\,m$, for some fixed $\eta >0$.
It was shown in \cite{Waldspurger2018} that that the Burer-Monteiro method might have spurious local minima when $\tau(p) \!+\! p \leq m$.
Hence, our required bound on~$p$ is close to the smallest that is needed to guarantee global optimality.

\Cref{thm:complete,thm:ls-complete} also make a compactness assumption on the domain of~\eqref{eq:bm}.
We point out that compactness is a standard assumption for the analysis of polynomial time methods, particularly those relying on smoothness constants.
It is possible to relax this assumption by requiring instead that the objective function is coercive on the domain,
since this ensures that all iterates remain bounded.

Finally, \Cref{thm:complete} assumes that LICQ holds for \eqref{eq:bm} everywhere on its domain.
The LICQ condition for the Burer-Monteiro problem is, in fact, equivalent to the \emph{primal nondegeneracy} of the original~SDP,
see \cite[Prop.6.2]{Boumal2018}.
Nondegeneracy is a well studied notion in the SDP literature~\cite{Chan2008,Alizadeh1997},
which plays a pivotal role in several techniques for solving SDPs
(besides the Burer-Monteiro method).
It is known that nondegeneracy is satisfied generically~\cite{Alizadeh1997},
i.e., every feasible $X$ is primal nondegenerate for generic $\mathcal{A},b,C$.
Moreover, the same holds even if $b,C$ are fixed and only $\mathcal{A}$ is generic, see~\cite[Prop.1]{Cifuentes2019}.
Hence, our LICQ assumption is satisfied for ``most'' problems, in a suitable sense.

\subsubsection*{Underlying local optimization method}

\Cref{thm:complete} relies on \Cref{alg:cartis}, a variant of the method from~\cite{Cartis2018}, to solve the Burer-Monteiro problem.
Such an algorithm was designed for theoretical purposes,
and will likely not perform well in practice.
Nonetheless, the theorems proved in this paper can accommodate other local optimization methods for solving the Burer-Monteiro problem.
Concretely, our results can be applied with any local optimization method that satisfies an analogue of \Cref{thm:complexity},
i.e., that provably finds an AFAC point in polynomial time.

The Burer-Monteiro problem is often solved in practice using ALM.
In particular, this was the method of choice in the original work by Burer and Monteiro~\cite{Burer2003,Burer2005}.
This motivates the following open problem.

\begin{problem}
  If we solve problem~\eqref{eq:bm} using ALM,
  can we provably find an approximately optimal solution for~\eqref{eq:sdp} in polynomial time?
\end{problem}

To answer the above question, we need to show that ALM converges to an AFAC point in polynomial time.
A step in this direction was recently given by Sahin et al.~\cite{Sahin2019}.
They proved that ALM computes a point satisfying the 2nd-order condition for the augmented Lagrangian function.
However, this does not imply that the 2nd-order condition holds for the original problem.

\subsection*{Acknowledgments}
Ankur Moitra was supported in part by a Microsoft Trustworthy AI Grant, NSF CAREER Award CCF-1453261, NSF Large CCF-1565235, a David and Lucile Packard Fellowship, an Alfred P.\ Sloan Fellowship and an ONR Young Investigator Award.
\appendix
\section{Computing AFAC points}\label{s:outer}

In this appendix we prove \Cref{thm:complexity},
which states that AFAC points can be computed in polynomial time.

\subsection{The algorithm}

Cartis et al.~\cite{Cartis2018} proposed a method for computing $q$-th order critical points for $q\in \{1,2,3\}$.
However, they use a nonstandard notion of criticality
which is not easy to translate into our setting.
We present here a slight modification of this algorithm that accommodates more general criticality conditions.

We focus on the constrained optimization problem~\eqref{eq:nonlinear}.
Consider the least squares functions
\begin{align*}
\nu(y) \,:=\, \|h(y)\|^2,
\qquad
\mu(t,y) \,:=\, (f(y){-}t)^2 + \|h(y)\|^2.
\end{align*}
We denote $\mu_t = \mu(t,\cdot)$ the function obtained by fixing the value of~$t$.
\Cref{alg:cartis} below is a variant of the method from~\cite{Cartis2018}.
It consists of two phases.
The first phase attempts to find an approximately feasible solution through the unconstrained problem $\min_y \nu(y)$.
If successful, the second phase minimizes~$f$ while preserving feasibility.
To do so, it solves a sequence of problems $\min_y \mu(t_k,y)$,
where the values $\{t_k\}_{k\geq 0}$ are decreasing.

\begin{algorithm}
  \caption{Constrained optimization algorithm based on~\cite{Cartis2018} }
  \label{alg:cartis}
  \begin{algorithmic}
    \Require{Initial point $y_0 \!\in\! \RR^n$, tolerances $\epsilon_0 \!\in\! \RR_+$, $\boldsymbol\epsilon \!\in\! \RR_+^q$, constant $\delta \!\in\! (0,1)$.}
    \Ensure{A point $y \in \RR^n$ and a number $t \leq f(y)$.}
    \State \centerline{\textsc{Phase I}}
    \State $y_1 := \mathrm{local\,min}_y \, \nu(y)$
    starting with $y_0$
    \State $t_0 := f(y_1)$
    \If{$\nu(y_1)>(\delta\epsilon_0)^2$}
    \;\Return $(y_1,t_0)$ \EndIf
    \State \centerline{\textsc{Phase II}}
    \State $t_1 := f(y_1) -(\epsilon_0^2 \!-\! \nu(y_1))^{1/2}$
    \For {$k=2,3,4,\dots$}
    \State $y_{k} :=$ $\mathrm{local\,min}_y \, \mu(t_{k-1},y)$
    starting with $y_{k-1}$
    \If{$\mu(t_{k-1},y_{k})<(\delta\epsilon_0)^2$}
    \Comment{case (a)}
    \State $t_{k} := f(y_{k}) -(\epsilon_0^2 \!-\! \nu(y_{k}))^{1/2}$
    \If{$\boldsymbol\chi(\mu_{t_{k}},y_{k}) \leq \boldsymbol\epsilon$}
    \;\Return $(y_{k},t_{k})$ \EndIf
    \EndIf
    \If{$\mu(t_{k-1},y_{k})\geq(\delta\epsilon_0)^2$ \& $f(y_{k})<t_{k-1}$}
    \Comment{case (b)}
    \State $t_{k} := 2 f(y_{k}) - t_{k-1}$
    \If{$\boldsymbol\chi(\mu_{t_{k}},y_{k}) \leq \boldsymbol\epsilon$}
    \;\Return $(y_{k},t_{k})$ \EndIf
    \EndIf
    \If{$\mu(t_{k-1},y_{k})\geq(\delta\epsilon_0)^2$ \& $f(y_{k})\geq t_{k-1}$}
    \Comment{case (c)}
    \State \Return $(y_{k},t_{k})$, with $t_k := t_{k-1}$
    \EndIf
    \EndFor
  \end{algorithmic}
\end{algorithm}

\Cref{alg:cartis} relies on an \emph{inner method} for solving the unconstrained problem
$\min_y \psi(y)$,
where $\psi$ is either $\nu$ or $\mu_t \!=\! \mu(t,\cdot)$.
Given
$\boldsymbol\epsilon \!=\! (\epsilon_1, \dots, \epsilon_q) \!\in\! \RR_+^q$,
the inner method looks for a point $y$ such that
$ \boldsymbol\chi(\psi,y) \leq \boldsymbol\epsilon$,
for some \emph{criticality measure}
$ \boldsymbol\chi = (\chi_1,\dots, \chi_q)$.
We assume that the $j$-th component $\chi_j(\psi,y)$ only involves derivatives $\{\nabla^d \psi(y)\}_{d\leq j}$ up to order~$j$.
For instance, the AC-criticality condition from~\eqref{eq:conditionsunconstrainedapprox} corresponds to the case
\begin{align}\label{eq:criticality}
  \boldsymbol\chi^{\mathrm{AC}}(\psi,y)
  \,:=\,
  \left(\,\|\nabla \psi(y)\|, \, -\mineig (\nabla^2 \psi(y))\,\right).
\end{align}
Given an initial point $y^0$ and tolerances
$\boldsymbol\epsilon \!\in\! \RR_+^q$,
the inner method produces iterates~$\{y^i\}_{i=1}^N$.
We assume that the final point $y^N$ achieves these tolerances
and that the objective function decreases proportionately to~$N$:
\begin{gather}\label{eq:inner}
  \boldsymbol\chi(\psi,y^N) \leq \boldsymbol\epsilon
  \qquad\text{ and }\qquad
  \psi(y^0) - \psi(y^N) \geq N\, \kappa_{\psi}\, p(\boldsymbol\epsilon),
\end{gather}
for some constant $\kappa_\psi \!>\!0$ and some function~$p$.
Hence, the number of iterations~$N$ is proportional to $p(\boldsymbol\epsilon)^{-1}$.

The next theorem provides guarantees for \Cref{alg:cartis}.
Our proof closely follows that of~\cite[Thm.4.5]{Cartis2018}
but has the advantage that it applies to a general class of criticality measures,
as opposed to~\cite{Cartis2018}, which relies on a particular nonstandard measure of criticality.
However our complexity is larger than in~\cite{Cartis2018} by a factor of~$\epsilon_0^{-1}$.

\begin{theorem} \label{thm:cartis}
  Assume that:
  \begin{itemize}
    \item The inner method satisfies~\eqref{eq:inner} for the function~$\nu$ with constant~$\kappa_\nu$.
    \item The inner method satisfies~\eqref{eq:inner} for the function~$\mu_t$,
      and the constant $\kappa_\mu$ is independent of~$t$.
    \item There exists $\beta \!>\! \epsilon_0$ and $f_{\mathrm{low}} \!\in\! \RR$
      such that $f(y) \!\geq\! f_{\mathrm{low}}$ for all $y \in \mathfrak{M}_\beta$,
      where $\mathfrak{M}_\beta := \{y: \|h(y)\| {\leq} \beta\}$.
  \end{itemize}
  Then the total number of inner iterations made in \Cref{alg:cartis} is at most
  \begin{align}\label{eq:numiters}
    p(\boldsymbol\epsilon)^{-1}
    \left( \kappa_\nu^{-1}\, \nu(y_0) \,+\,
      \epsilon_0\, \kappa_\mu^{-1}\, (1{-}\delta)^{-1}\,
      (f(y_1){-}f_{\mathrm{\mathrm{low}}}{+}\beta) \right),
  \end{align}
  and the algorithm returns a pair $(y,t)$ such that:
  \begin{subequations} \label{eq:cartis}
  \begin{align}
    \label{eq:cartisgood}
    &\text{either }
    &&t < f(y),
    &&\|h(y)\| \leq \epsilon_0,
    &&\boldsymbol\chi(\mu_t,y) \leq \boldsymbol\epsilon,
    \\
    \label{eq:cartisbad}
    &\text{or }
    &&t = f(y),
    &&\|h(y)\| > \delta \epsilon_0,
    &&\chi_1(\nu,y) \leq \epsilon_1.
  \end{align}
  \end{subequations}
\end{theorem}

\subsection{Proof of \Cref{thm:cartis}}

Let $K$ be the number of outer iterations of \Cref{alg:cartis}.
Consider the sets of indices:
\begin{align*}
  A &\,:=\, \{1\} \cup \{ k : 2\!\leq\! k \!\leq\! K \text{ and case (a) is applied} \}, \\
  B &\,:=\, \{ k : 2\!\leq\! k \!\leq\! K \text{ and case (b) is applied} \}.
\end{align*}
The following lemma gives a few properties of \Cref{alg:cartis}.
Its proof is identical to \cite[Lem.3.1]{Cartis2018}.

\begin{lemma}
  If the algorithm reaches Phase II, then:
  \begin{align}
    \label{eq:property0}
    &\nu(y_k) \leq \mu(t_k,y_k) \leq \epsilon_0^2,
    &&\kern-1em0 \leq f(y_k) - t_k \leq \epsilon_0,
    &&\text{for } k \geq 1,
    \\
    \label{eq:propertyA}
    &\mu(t_k,y_{k}) = \epsilon_0^2,
    &&\kern-1emt_{k-1} - t_{k} \geq (1{-}\delta)\epsilon_0,
    &&\text{for } k\in A,
    \\
    \label{eq:propertyB}
    &\mu(t_{k},y_{k}) = \mu(t_{k-1},y_{k}),
    &&\kern-1emt_{k-1} > t_k,
    &&\text{for } k\in B,
    \\
    \label{eq:propertyK}
    &\mu(t_{k},y_{k}) \geq (\delta\epsilon_0)^2,
    &&\kern-1em\boldsymbol\chi(\mu_{t_k},y_k) \leq \boldsymbol\epsilon,
    &&\text{for }k=K.
  \end{align}
\end{lemma}

Let $(y,t)$ be the output of \Cref{alg:cartis},
and let us show~\eqref{eq:cartis}.
Assume first that the algorithm terminates in Phase~I.
Then $y$ is a local minimum of~$\nu$,
$\nu(y) \!>\! (\delta \epsilon_0)^2$,
and $t \!=\! f(y)$.
Hence \eqref{eq:cartisbad} holds.
Assume now that the algorithm terminates in Phase~II.
By \eqref{eq:property0} and \eqref{eq:propertyK},
we have
\begin{align*}
t \leq f(y),\qquad
(\delta\epsilon_0)^2 \leq \mu_t(y) \!\leq\! \epsilon_0^2,\qquad
\boldsymbol\chi(\mu_t,y) \leq \boldsymbol\epsilon.
\end{align*}
If $f(y) \!<\! t$ then $\|h(y)\| \!\leq\! \sqrt{\mu_t(y)} \!\leq\! \epsilon_0$,
so \eqref{eq:cartisgood} holds.
Consider now the case that $f(y)\!=\!t$.
Note that $\mu_t(y) \!=\! \nu(y)$, $\nabla \mu_t(y) \!=\! \nabla \nu(y)$.
Then $\chi_1(\mu_t,y) \!=\! \chi_1(\nu,y)$,
as they only involve derivatives up to order~1.
Since $\|h(y)\| \!=\! \sqrt{\mu_t(y)} \!\geq\! \delta \epsilon_0$,
then \eqref{eq:cartisbad} holds.

We proceed to show that the number of inner iterations is bounded by~\eqref{eq:numiters}.
Each outer iteration $k$ of \Cref{alg:cartis} calls the inner method once.
Let $N_k$ be the number of inner iterations made in this call.
The total number of inner iterations is
$ \sum\nolimits_{k=1}^K N_k.$

We first analyze Phase~I.
The inner method is applied to the problem $\min_y \nu(y)$,
starting with $y_0$ and terminating with $y_1$.
By \eqref{eq:inner}, we have
\begin{align*}
  \nu(y_0)
  \,\geq\,
  \nu(y_0) - \nu(y_1)
  \,\geq\,
  N_1\, \kappa_\nu\, p(\boldsymbol\epsilon).
\end{align*}
It follows that
$N_1 \leq \nu(y_0) /\kappa_\nu p(\epsilon)$.

We proceed to Phase~II.
For each $a \in A$, let $n(a)$ be the next integer that lies in $A$.
For the largest $a \in A$ we define $n(a) := K$,
where $K$ is the final iteration.
We can group the indices $k \geq 2$ as follows:
\begin{align*}
  \{2,3,\dots,K\} = \bigcup_{a \in A} K_a,
  \quad\text{ where }\quad
  K_a := \{a{+}1,a{+}2, \dots , n(a)\}.
\end{align*}

We will show that for any $a\in A$ we have that
\begin{align}\label{eq:boundKa}
  N(K_a)
  \,:=\,
  \sum_{k \in K_a} N_k
  \,\leq\,
  \epsilon_0^2 /\kappa_\mu p(\epsilon).
\end{align}
Consider an iteration $k \in K_a$.
The inner method is applied to $\min_y \mu(t_{k-1},y)$,
starting with $y_{k-1}$ and terminating with $y_k$.
By~\eqref{eq:inner}, we have
\begin{align*}
  \mu(t_{k-1},y_{k-1}) - \mu(t_{k-1},y_k)
  \,\geq\,
  N_k\, \kappa_\mu\, p(\boldsymbol\epsilon).
\end{align*}
Observe that $K_a \setminus \{n(a)\} \subset B$.
By \eqref{eq:propertyB}, we have
\begin{align*}
  \mu(t_{k-1},y_k) = \mu(t_{k},y_k)
  \quad\text{ for }\quad
  k \in K_a\setminus \{n(a)\}.
\end{align*}
Also note that $\mu(t_{a},y_a) = \epsilon_0^2$ by \eqref{eq:propertyA}.
Therefore,
\begin{align*}
  \epsilon_0^2
  \,&\geq\,
  \mu(t_{a},y_{a}) - \mu(t_{n(a)}{-}1,y_{n(a)})\\
  &=\,
  \sum_{k \in K_a} \mu(t_{k-1},y_{k-1}) - \mu(t_{k-1},y_k)
  \,\geq\,
  \sum_{k \in K_a} N_k \,\kappa_\mu\, p(\boldsymbol\epsilon).
\end{align*}
By rearranging the above inequality we get~\eqref{eq:boundKa}.

Let us now upper bound the cardinality of~$A$.
By~\eqref{eq:propertyA} and~\eqref{eq:propertyB} we have that
$t_{k-1} \!-\! t_{k}$ is at least $(1{-}\delta)\epsilon_0$ for $k\in A$,
and is positive for $k \in B$.
Also note that $t_0 \!=\! f(y_1)$
and $t_K \!\geq\! f(y_K) {-} \epsilon_0 \!\geq\! f_{\mathrm{low}} {-} \beta$
by~\eqref{eq:property0}.
Then
\begin{align*}
  f(y_1) {-} f_{\mathrm{\mathrm{low}}} {+} \beta
  \,\geq\,
  t_0 {-} t_{K}
  \,=\,
  \sum_{k=1}^K (t_{k-1} {-} t_{k})
  \,\geq\,
  \sum_{k\in A} (t_{k-1} {-} t_{k})
  \,\geq\,
  |A| \, (1{-}\delta)\epsilon_0,
\end{align*}
and hence
$|A| \leq (f(y_1){-}f_{\mathrm{\mathrm{low}}}{+}\beta)/(1{-}\delta)\epsilon_0$.

Combining everything, we derive
\[
  \sum_{k=1}^K N_k
  \,\leq\,
  N_1 + |A| \cdot \max_{a\in A} N(K_a)
  \,\leq\,
  \frac{\nu(y_0)}{\kappa_\nu p(\epsilon)}
  + \frac{f(y_1){-}f_{\mathrm{\mathrm{low}}}{+}\beta}{(1{-}\delta)\epsilon_0}
  \cdot \frac{\epsilon_0^2}{\kappa_\mu\, p(\epsilon)},
  \qedhere
\]
which is equal to~\eqref{eq:numiters}.

\subsection{Proof of \Cref{thm:complexity}}

We finally show that AFAC points can be computed in polynomial time.
Let
$\varepsilon_0,\varepsilon_1,\varepsilon_2,\gamma,R_\lambda$
be as in the statement of \Cref{thm:complexity}.
We consider \Cref{alg:cartis} with parameters
\begin{gather*}
  \delta \,:=\, 1/2,
  \qquad
  q \,:=\, 2,
  \qquad
  \boldsymbol\epsilon \,:=\, (\epsilon_1,\epsilon_2),
  \\
  \epsilon_0 \,:=\, \varepsilon_0,
  \quad
  \epsilon_1 \,:=\, R_\lambda^{-1}\, \varepsilon_0 \,\varepsilon_1,
  \quad
  \epsilon_2 \,:=\, \tfrac{1}{2} R_\lambda^{-1}\, \varepsilon_0\, \varepsilon_2.
\end{gather*}
For the inner method we use the ARC algorithm from \Cref{thm:complexityunconstrained},
using the criticality measure~\eqref{eq:criticality}.
\Cref{alg:cartis} returns a pair $(y,t)$.
The associated multiplier is
$\lambda := (f(y){-}t)^{-1}\, h(y) \in \RR^m$,
which is defined only if $f(y)\neq t$.

In order to apply \Cref{thm:cartis},
we have to check that the functions $\nu$ and $\mu_t \!=\! \mu(t,\cdot)$ are smooth enough
so that the inner algorithm satisfies~\eqref{eq:inner}.

\begin{lemma}[{\cite[Lem.4.1]{Cartis2018}}] \label{thm:lipchitz}
  Assume that $\{\nabla^j f\}_{j=0}^q, \{\nabla^j h\}_{j=0}^q$ are uniformly bounded and Lipschitz continuous on a set~$D \subset \RR^n$.
  Then
  \begin{enumerate}[label=(\roman*)]
    \item $\{\nabla^j \nu\}_{j=0}^q$ are uniformly bounded and Lipschitz continuous on~$D$.
    \item $\{\nabla^j \mu_t\}_{j=0}^q$ are uniformly bounded and Lipschitz continuous on $D \cap\nobreak B_{t}$,
      with $B_{t}:= \{y: |f(y){-}t| {\leq} 1\}$,
      and the constants are independent of~$t$.
  \end{enumerate}
\end{lemma}

The above lemma shows that $\nu$ is smooth on~$\mathfrak{M}_\beta$
and $\mu_t$ is smooth on~$\mathfrak{M}_\beta \cap\nobreak B_t$,
with $B_{t}:= \{y: |f(y){-}t| {\leq} 1\}$.
Note that all points $y_k$ produced by \Cref{alg:cartis} lie in $\mathfrak{M}_\beta \cap B_t$ because of~\eqref{eq:property0}.
Since $\nu, \mu_t$ are sufficiently smooth,
we can apply \Cref{thm:complexityunconstrained} (see also~\cite{Cartis2011}).
We conclude that the inner method satisfies~\eqref{eq:inner} with
\begin{align*}
  p(\boldsymbol\epsilon)
  \,=\,
  \min\{\epsilon_1^{2},\,\epsilon_2^{3}\}
  \,=\,
  \Omega(\min\{\varepsilon_0^2 \,\varepsilon_1^2,\, \varepsilon_0^3 \,\varepsilon_2^3\}).
\end{align*}
Hence, by \Cref{thm:cartis},
the total number of inner iterations is
$
  O(p(\boldsymbol\epsilon)^{-1}) =
  O(\max\{\varepsilon_0^{-2} \,\varepsilon_1^{-2},
  \varepsilon_0^{-3} \,\varepsilon_2^{-3}\}).
$
Since each inner iteration requires $O(1)$ function evaluations (see \Cref{thm:complexityunconstrained}),
then the total number of function evaluations has the same order of magnitude.

Let us see that the conditions~\eqref{eq:conditionsKKTapprox} hold.
Let $(y,t)$ be the output of \Cref{alg:cartis}.
By \Cref{thm:cartis}, this pair satisfies either \eqref{eq:cartisgood} or~\eqref{eq:cartisbad}.
Let us see that~\eqref{eq:cartisbad} cannot occur.
Assume that
\begin{align*}
  \|h(y)\| > \epsilon_0/2,
  \qquad
  \|\nabla \nu(y)\| \leq \epsilon_1.
\end{align*}
Observe that $\|h(y)\| \!\leq\! \epsilon_0 \!\leq\! \beta$ by \eqref{eq:property0},
and hence $\varrho$-LICQ holds at~$y$.
Then
\begin{align*}
  \varrho\, \epsilon_0
  \,<\,
  2\, \varrho \|h(y)\|
  \,\leq\,
  2\, \|h(y)^{\!T} \nabla h(y)\|
  \,=\,
  \|\nabla \nu(y)\|
  \,\leq\,
  \epsilon_1.
\end{align*}
Also note that that
\begin{align*}
  R_\lambda^{-1}
  \,\leq\,
  \tfrac{1}{2} \varrho(1{+}L_f)^{-1}
  \,\leq\,
  \tfrac{1}{2} \varrho,
  \qquad
  \epsilon_1/\epsilon_0
  \,=\,
  R_\lambda^{-1}\, \varepsilon_1
  \,\leq\,
  \tfrac{1}{2}\varrho\, \varepsilon_1
  \,\leq\,
  \tfrac{1}{2}\varrho.
\end{align*}
The last two equations give a contradiction.

Then the output $(y,t)$ satisfies~\eqref{eq:cartisgood}.
Hence, $t < f(y)$ and
\begin{align*}
  \|h(y)\| \,\leq\, \epsilon_0,
  \qquad
  \|\nabla \mu_t(y)\| \,\leq\, \epsilon_1,
  \qquad
  \nabla^2 \mu_t(y) \,\succeq\, -\epsilon_2 \id_n.
\end{align*}
Let
$ \alpha := (f(y){-}t)^{-1} $,
so that
$ \lambda = \alpha\, h(y). $
It can be checked that
$\alpha^2 \mu_t(y) = \|(1,\lambda)\|^2$.
Note that $\mu_t(y) \geq (\epsilon_0/2)^2$ by \eqref{eq:propertyK},
and hence
\begin{align*}
  \alpha
  \,=\,
  \mu_t(y)^{-1/2}\,\|(1,\lambda)\|
  \,\leq\,
  2\, \varepsilon_0^{-1} \|(1,\lambda)\|.
\end{align*}
The Lagrangian function
$L(y,\lambda) \!=\! f(y) \!+\! \lambda \!\cdot\! h(x)$
is closely related to~$\mu_t(y)$.
A simple calculation gives that
\begin{align}\label{eq:Lagmu}
  \nabla \Lag(y,\lambda)
  \,=\, \alpha \cdot \tfrac{1}{2} \nabla\mu_t(y),
  \qquad
  \nabla^2 \Lag(y,\lambda)
  \,=\, \alpha\,\bigl(\tfrac{1}{2}\nabla^2 \mu_t(y)
  \!-\! \tilde{J}^{\,T} \tilde{J} \bigr),
\end{align}
where
$\tilde J := \left(\begin{smallmatrix} \nabla f(y) \\ \nabla h(y) \end{smallmatrix}\right)$
is the augmented Jacobian.

We proceed to verify~\eqref{eq:firstorderapprox}.
We already have that
$ \|h(y)\| \leq \varepsilon_0.$
Note that
\begin{gather}\label{eq:nablaLag}
  \|\nabla \Lag(y,\lambda) \|
  \,=\,
  \tfrac{1}{2}\, \alpha\, \|\nabla \mu_t(y)\|
  \,\leq\, \varepsilon_0^{-1} \|(1,\lambda)\| \,\epsilon_1\,
  \,=\, R_\lambda^{-1} \|(1,\lambda)\|\, \varepsilon_1.
\end{gather}
We claim that $\|(1,\lambda)\| \!\leq\! R_\lambda$.
By \Cref{thm:lambda} and
$R_\lambda^{-1} \!\leq\! \varrho/2$, $\varepsilon_1 \!\leq\! 1$,
we have
\begin{gather*}
  \|\lambda\|
  \,\leq\,
  \varrho^{-1}\,(R_\lambda^{-1}\, \varepsilon_1\|(1,\lambda)\| + \|\nabla f(y)\|)
  \,\leq\,
  \tfrac{1}{2}(1{+}\|\lambda\|) + \varrho^{-1} L_f.
\end{gather*}
It follows that
$\|\lambda\| \leq 1 \!+\! 2 \varrho^{-1} L_f = R_\lambda \!-\! 1$
and hence $\|(1,\lambda)\| \leq R_\lambda$,
as we claimed.
Then $\|\nabla \Lag(y,\lambda)\| \leq \varepsilon_1$ by~\eqref{eq:nablaLag}.

We now verify~\eqref{eq:secondorderapprox}.
Let $u \!\in\! \RR^n$ of unit norm such that $\|J u\| \!\leq\! \gamma$,
where $J \!:=\! \nabla h(y)$.
We need to show that $u^T \nabla^2 \Lag(y,\lambda) u \geq -\varepsilon_2$.
By \eqref{eq:Lagmu}, we have
\begin{align}\label{eq:nabla2Lag}
  u^T \nabla^2 \Lag(y,\lambda) u
  \,=\,
  \alpha\,(\tfrac{1}{2} u^T \nabla^2 \mu_t(y) u - \|\tilde{J} u\|^2).
\end{align}
Note that
$u^T \nabla^2 \mu_t(y) u
\geq -\epsilon_2
= -\tfrac{1}{2} R_\lambda^{-1} \varepsilon_0 \varepsilon_2$.
We bound $\|\tilde J u\|$ next:
\begin{gather*}
  \tilde J
  \,=\,
  \begin{pmatrix} \nabla f(y) \\ \nabla h(y) \end{pmatrix}
  \,=\,
  \begin{pmatrix} \nabla \Lag(y,\lambda) \\ 0 \end{pmatrix}
  + \begin{pmatrix} -\lambda^T J \\ J \end{pmatrix},
  \\
  \|\tilde J u \|
  \,\leq\,
  \|\nabla \Lag(y,\lambda)\|
  + \|(1,\lambda)\|\, \|J u\|
  \,\leq\,
  \varepsilon_1 + \gamma \, \|(1,\lambda)\|
  \,\leq\,
  \tfrac{1}{2} (R_\lambda^{-1} \,\varepsilon_0 \,\varepsilon_2)^{1/2},
\end{gather*}
where we used that $\varepsilon_1$ and $\gamma R_\lambda$
are at most $\tfrac{1}{4} (R_\lambda^{-1} \varepsilon_0 \varepsilon_2)^{1/2} $
by~\eqref{eq:epsilonineqs}.
Hence
\begin{align*}
  \alpha\,(\tfrac{1}{2} u^T \nabla^2 \mu_t(y) u - \|\tilde{J} u\|^2)
  \,\geq\,
  -(2\varepsilon_0^{-1} \|(1,\lambda)\|) \cdot
  (\tfrac{1}{2} R_\lambda^{-1} \varepsilon_0\, \varepsilon_2)
  \,\geq
  -\varepsilon_2.
\end{align*}
Together with~\eqref{eq:nabla2Lag},
we get that $u^T \nabla^2 L(y,\lambda) u \geq -\varepsilon_2$.

\bibliographystyle{abbrv}
\bibliography{../arxiv/refs.bib}

\end{document}